\documentclass{amsart}

\usepackage{enumerate, amscd, amssymb, mathtools}

\newcommand\ede{ \, := \, }
\newcommand\seq{ \, = \, }

\newcommand\datver[1]{\, \def\datverp
{\par\boxed{\boxed{\text{Version: #1; Run: \today}}}}}\datver{0.1}

\newcommand{\supp}{\operatorname{supp}}

\newcommand\pullback{\sp{\downarrow\downarrow}}
\newcommand{\tto}{\rightrightarrows}
\newcommand\mathbfPsi{\mathbf \Psi}

\newcommand{\Prim}{\operatorname{Prim}}
\newcommand{\End}{\operatorname{End}}
\newcommand{\ess}{\mathrm{ess}}

\newcommand{\CC}{\mathbb C}

\newcommand{\RR}{\mathbb R}

\newcommand{\ZZ}{\mathbb Z}

\newcommand{\CI}{{\mathcal C}^{\infty}}
\newcommand{\CIc}{{\mathcal C}^{\infty}_{\text{c}}}
\newcommand\pa{{\partial}}
\newcommand\Cstar{C\sp{\ast}}

\newcommand\Cs[1]{C\sp{\ast}(#1)}
\newcommand\rCs[1]{C_r\sp{\ast}(#1)}

\newcommand{\Aut}{\operatorname{Aut}}
\newcommand{\Diff}{\operatorname{Diff}}

\newcommand\ssub{stratified submersion}
\newcommand\Lie{\operatorname{Lie}}

\newcommand{\maA}{\mathcal A}

\newcommand{\maC}{\mathcal C}

\newcommand{\maF}{\mathcal F}
\newcommand{\maG}{\mathcal G}
\newcommand{\maH}{\mathcal H}

\newcommand{\maK}{\mathcal K}
\newcommand{\maL}{\mathcal L}
\newcommand{\maM}{\mathcal M}
\newcommand{\maP}{\mathcal P}
\newcommand{\maR}{\mathcal R}
\newcommand{\maS}{\mathcal S}

\newcommand{\maV}{\mathcal V}
\newcommand{\maW}{\mathcal W}

\newcommand{\de}{{\rm d}}

\def\ccinf0{{\mathcal C}_c^{\infty,0}}

\newtheorem{theorem}{Theorem}[section]
\newtheorem{proposition}[theorem]{Proposition}
\newtheorem{corollary}[theorem]{Corollary}

\newtheorem{lemma}[theorem]{Lemma}
\newtheorem{notation}[theorem]{Notations}
\theoremstyle{definition}
\newtheorem{definition}[theorem]{Definition}
\theoremstyle{remark}
\newtheorem{remark}[theorem]{Remark}
\newtheorem{example}[theorem]{Example}

\begin{document}

\title[{Fredholm conditions on non-compact manifolds}]{Fredholm
  conditions on non-compact manifolds: theory and examples}

\author[C. Carvalho]{Catarina Carvalho} \address{{Dep. Matem\'{a}tica,
    Instituto Superior T\'{e}cnico, University of Lisbon, Av. Rovisco
    Pais, 1049-001 Lisbon, Portugal }}
\email{catarina.carvalho@math.tecnico.ulisboa.pt}

\author[V. Nistor]{Victor Nistor} \address{Universit\'{e} de Lorraine,
  UFR MIM, Ile du Saulcy, CS 50128, 57045 METZ, France and
%
% Pennsylvania State University, Math. Dept., University Park, PA 16802,
% USA}
%
  Inst. Math. Romanian Acad.  PO BOX 1-764, 014700 Bucharest Romania}
\email{victor.nistor@univ-lorraine.fr}

\author[Y. Qiao]{Yu Qiao} \address{School of Mathematics and
  Information Science,\\ Shaanxi Normal University, Xi'an, 710119,
  China} \email{yqiao@snnu.edu.cn}

\thanks{V.N. has been partially supported by ANR-14-CE25-0012-01. Qiao
  was partially supported by NSF of China (11301317,
  11571211). Carvalho was partially supported by Fundação para a
  Ciência e Tecnologia UID/MAT/04721/2013 (Portugal).\\
Manuscripts available from {\bf http:{\scriptsize
    //}iecl.univ-lorraine.fr{\scriptsize
    /}$\tilde{}$Victor.Nistor{\scriptsize /}}}



% DO NOT DELETE
% AMS Subject classification (2010):
% 22A22 Topological groupoids (including differentiable and Lie groupoids) [See also 58H05]
% 35 = PDEs
% 45B05 = Fredholm integral equations
% 46 = Functional Analysis
% 46L = self-adjoint algebras
% 46L05 general theory of C* algebras
% 46L60 Applications of selfadjoint operator algebras to physics
% 46L80  $K$-theory and operator algebras (including cyclic theory) [See also 18F25, 19Kxx, 46M20, 55Rxx, 58J22]
% 47 = Operator theory
% 47A53 = (Semi-) Fredholm operators; index theories
% 47G30 = pseudodifferential operators
% 47L80 Algebras of specic types of operators
  % (Toeplitz, integral, pseudodifferential, etc.)
% 46L87 = Noncommutative geometry
% 47L90 Applications of operator algebras to physics
% 58 = global analysis, analysis on manifolds
% 58H = Pseudogroups, differentiable groupoids and
  % general structures on manifolds
% 58H05 = Pseudogroups and differentiable groupoids (the only
  % one acceptable in the 58H)
% 58J = Partial differential equations on manifolds; differential operators
% 58J40 = Pseudodifferential and Fourier integral operators on manifolds

\date\today

%%%%%%%%%%%%%%%%%%%%%%%%%%%%%%%%%%%%%%%%%
%%%%%%%%%%%%%%%%%%%%%%%%%%%%%%%%

\begin{abstract}
We give explicit Fredholm conditions for classes of pseudodifferential
operators on suitable singular and non-compact spaces. In particular,
we include a ``user's guide'' to Fredholm conditions on particular
classes of manifolds including asymptotically hyperbolic manifolds,
asymptotically Euclidean (or conic) manifolds, and manifolds with
poly-cylindrical ends. The reader interested in applications should be
able read right away the results related to those examples, beginning
with Section \ref{sec.ex}.  Our general, theoretical results are that
an operator adapted to the geometry is Fredholm if, and only if, it is
elliptic and all its \emph{limit operators}, in a sense to be made
precise, are invertible. Central to our theoretical results is the
concept of a {\em Fredholm groupoid}, which is the class of groupoids
for which this characterization of the Fredholm condition is valid.
We use the notions of \emph{exhaustive} and \emph{strictly spectral
  families of representations} to obtain a general characterization of
Fredholm groupoids. In particular, we introduce the class of the
so-called \emph{groupoids with Exel's property} as the groupoids for
which the regular representations are exhaustive. We show that the
class of ``\ssub\ groupoids'' has Exel's property, where
\ssub\ groupoids are defined by glueing fibered pull-backs of bundles
of Lie groups. We prove that a \ssub\ groupoid is Fredholm whenever
its isotropy groups are amenable.  Many groupoids, and hence many
pseudodifferential operators appearing in practice, fit into this
framework. This fact is explored to yield Fredholm conditions not only
in the above mentioned classes, but also on manifolds that are
obtained by desingularization or by blow-up of singular sets.
\end{abstract}

\maketitle

\tableofcontents

\section{Introduction}

We obtain in this paper necessary and sufficient conditions for
classes of operators to be Fredholm. Our results specialize to yield
Fredholm conditions for pseudodifferential operators on manifolds with
cylindrical and poly-cylindrical ends, on manifolds that are
asymptotically Euclidean, and on manifolds that are asymptotically
hyperbolic. Other examples of non-compact manifolds covered by our
results include operators obtained by desingularization of suitable
singular spaces by successively blowing up the lowest dimensional
singular strata.

\subsection{{Background and main result}}
Let $M_0$ be a Riemannian manifold and let $P\colon H\sp{s}(M_0; E)
\to H\sp{s-m}(M_0; F)$ be an order $m$ pseudo-differential operator
acting between Sobolev sections of two smooth, Hermitian vector
bundles $E, F$. Recall that $P$ is called {\em elliptic} if its
principal symbol $\sigma_m(P)$ is invertible outside the zero section.
When $M_0$ is \emph{compact,} a classical, well-known result
\cite{CordesHerman68,Seeley59,Seeley63} states that $P$ is Fredholm
if, and only if, it is elliptic.  This classical Fredholm result has
many applications, so a natural question to ask is to what extent it
extends to (suitable) non-compact manifolds.  The example of constant
coefficient differential operators on ${\RR\sp{n}}$ shows that that
this classical result might be no longer true as stated if $M_0$ is
not compact.

On certain classes of manifolds, however, it is possible to
reformulate this classical result as follows. Let $M_0$ be a
non-compact manifold with ``amenable ends'' (see Subsection
\ref{ssec.pseudodifferential} for the precise definition). Then we can
associate to $M_0$ the following data:
\begin{enumerate}[(1)]
 \item Smooth manifolds $M_\alpha$, $\alpha \in I$, and
  \item Lie groups $G_\alpha$ acting on $M_\alpha$, $\alpha \in I$,
\end{enumerate}
(for a suitable index set $I$), satisfying the following theorem.

\begin{theorem}\label{thm.nonclassical}
Let $P$ be an order $m$ pseudodifferential operator on $M_0$
compatible with the geometry.  Then one can associate to $P$ certain
$G_\alpha$-invariant pseudodifferential operators $P_\alpha$ on
$M_\alpha$ with the following property:
 \begin{multline*} %\label{eq.Fredholm.c}
	P : H^s(M_0; E) \to H^{s-m}(M_0; F) \mbox{ is Fredholm}
        \ \ \Leftrightarrow \ \ P \mbox{ is elliptic and }\\
	\ P_{\alpha} { : H^s(M_\alpha; E) \to H^{s-m}(M_\alpha; F)}\
	\mbox{ is invertible for all } \alpha \in I \,.
\end{multline*}
\end{theorem}

Many results of this kind were obtained for psedudodifferential
operators before. See for instance \cite{DLR, MantoiuReine,
  MelrosePiazza, MeloNestSchrohe, SchSch1, SchulzeBook91} and the
references therein for a very small sample.  Similar results for {\em
  differential} operators were also obtained earlier, see
\cite{daugeBook, Kondratiev67, KMR, NP, PlamenevskiBook} and the
references therein, again for a very small sample.  Other related
results appeared in the context of localization principles, often in
relation to the $N$-body problem. See \cite{DamakGeorgescu, GI2,
  GeorgescuNistor2, Mantoiu2, Mantoiu1, RochBookLimit, RochBookNGT}
and the references therein.  Inspired by these works, we shall call
the operators $P_\alpha$ {\em limit operators} (of $P$). We shall
refer to results similar to Theorem \ref{thm.nonclassical} as {\em
  non-local Fredholm conditions}.

The class of ``manifolds with amenable ends'' is introduced such that,
almost by definition, it is close to being the largest class of
manifolds for which Theorem \ref{thm.nonclassical} is valid.  The
challenge then becomes to provide a large enough class of manifolds
with amenable ends, which we do in this paper.  The term ``amenable''
comes from the fact that certain isotropy groups at infinity must be
amenable.  In many cases, the manifolds $M_\alpha$ are obtained from
the orbits of certain natural vector fields acting on a
compactification $M$ of $M_0$ and from their isotropies. This is the
case for manifolds with cylindrical and poly-cylindrical ends, for
manifolds that are asymptotically Euclidean, and for manifolds that
are asymptotically hyperbolic, which are all particular cases of
manifolds with amenable ends. See the last section of the paper
for explicit statements covering in detail the case of these classes
of manifolds.  It is interesting to notice that the case of
asymptotically hyperbolic manifolds leads to the study of certain
operator on solvable Lie groups.

Theorem \ref{thm.nonclassical} is almost a consequence of the results
in \cite{LMN1, LNGeometric}, but the results of those papers turned
out to be difficult to use by non-specialists. We have thus tried in
the last section of the paper to provide a presentation such that
the reader can {\em understand the main results without any knowledge
  of groupoids or $\Cstar$-algebras.}

\subsection{Exhaustive families of representations}
It was realized by some of the authors mentioned above that the
non-local Fredholm conditions of Theorem \ref{thm.nonclassical} are
related to the representation theory of certain $C^{\ast}$-algebras
$\mathfrak{A}$. See also \cite{CordesBook, CordesHerman68,
  PlamenevskiBook,TaylorGlf} and the references therein.  These
$C\sp{\ast}$-algebras are such that they contain the operators of the
form $a = (1 + \Delta)\sp{(s-m)/2} D (1 + \Delta)\sp{-s/2}$ (where
$\Delta$ is the positive Laplacian) acting on $L\sp{2}(M_0)$. The
reason for considering the operator $a$ is that $D$ is Fredholm if,
and only if, $a$ is Fredholm; moreover, $a$ is bounded.  In most
practical applications, these $C^{\ast}$-algebras can be chosen to be
groupoid $C^{\ast}$-algebras. These ideas are used in the proof of
Theorem \ref{thm.nonclassical2}, which contains Theorem
\ref{thm.nonclassical} as a particular case.  The operators
$P_{\alpha}$ in Theorem \ref{thm.nonclassical} then can be obtained as
homomorphic images of the operator $P$, that is $P_\alpha=
\pi_\alpha(P)$.

The relevant families of representations in this setting are those of
\emph{strictly spectral (and strictly norming)} families of
representations, introduced by Roch in \cite{Roch} and recently
studied in detail by Nistor and Prudhon in \cite{nistorPrudhon}. In
that paper, the concept of an \emph{exhaustive family} of
representations has emerged as a useful concept to study strictly
spectral families of representations. From a technical point of view,
these families of representations form the backbone of the theoretical
results in this paper.

\subsection{Fredholm Lie groupoids}
The key point of our approach is to start with a general study of
Fredholm conditions for pseudodifferential operators in the framework
of Fredholm Lie groupoids.  A {\em Fredholm Lie groupoid} $\maG$ is,
by definition, a Lie groupoid for which the Fredholm property of its
associated operators is equivalent to the invertibility of the
principal symbol and of its fiberwise boundary restrictions. More
precisely, if $M$ is the set of units of $\maG$ and if $\maG$ is the
pair groupoid above $M_0 := M \smallsetminus \pa M$, we have that
$\maG$ is a {\em Fredholm Lie groupoid} if it has the following
property:
\begin{equation}\label{eq.main.cond}
 \mbox{``$a \in 1 + \Cstar_r(\maG)$ is
   Fredholm\ $\Leftrightarrow$\ $\pi_x(a)$ is invertible for all $x
   \in \pa M$.''}
\end{equation}
The reader will recognize in condition \eqref{eq.main.cond} the type
of Fredholm conditions that appeared in Theorem \ref{thm.nonclassical}
(and which are typically used in practice). It turns out that if the
defining condition \eqref{eq.main.cond} is satisfied (for all $a \in 1
+ \Cstar_r(\maG)$) then it will be satisfied for many other operators
associated to $\maG$, in particular it will be satisfied for $a =P \in
\Psi^m(\maG; E, F)$, a pseudodifferential operators on $\maG$ acting
between the sections of the vector bundles $E, F \to M$ (however, for
pseudodifferential operators, one has to add the ellipticity
condition).

The extension of the Fredholm conditions in Equation
\eqref{eq.main.cond} from operators in $1 + \Cstar_r(\maG)$ to
operators in $\Psi^m(\maG; E, F)$ turns out to be almost automatic and
is based on the use of strictly spectral and exhaustive families of
representations of groupoid $\Cstar$-algebras (see Section
\ref{sec.exhaustive} for definitions and references).  Thus, in our
paper, we shift the study of the Fredholm conditions from the study of
a {\em single} operator to the study of suitable {\em algebras}
containing it.  In particular, it will be enough to study the
properties and the representations of {\em regularizing}
pseudodifferential operators. The Fredholm properties of higher order
pseudodifferential operators will then follow simply by including the
ellipticity condition.  Thus the Fredholm conditions established in
this paper will be formulated in terms of representations of groupoid
$C\sp{\ast}$-algebras.

We give various characterizations of Fredholm groupoids and provide
methods to prove that individual groupoids are Fredholm. We introduce
the class of \ssub\ groupoids, which is built out by glueing fibered
pull-backs of bundles of Lie groups. We show that a \ssub\ groupoid
with amenable stabilizers is Fredholm 
and hence that it yields a manifold with amenable ends.

\subsection{Examples and applications}

The main significance of the class of \ssub\ Lie groupoids is that
many of the groupoids that appear in practice exhibit this structure
naturally.  Typically, using the notations as above, we have that $M$
identifies with a compactification of $M_0$ to a manifold with
corners, where the behavior at $F=M\setminus M_0$ models the behavior
at infinity (`at the ends'). Often the corresponding isotropy groups
are amenable.  A class of manifolds that fits this perspective is the
class of Lie manifolds \cite{LMN1, LNGeometric}, a few examples of
which appear in Section \ref{sec.ex}.

Moreover, \ssub\ Lie groupoids are also tailored to applications to
singular spaces obtained by (iterative) desingularization procedures,
desingularization being the analog in the category of groupoids of the
blow-up construction in the category of manifolds with corners.  In
fact, it can be seen that desingularization preserves Fredholm
groupoids, and that the class of \ssub\ groupoids is closed under
desingularization.  While we do not pursue this approach here in full
generality, we outline the key ideas involved and give the
construction of such a groupoid yielding Fredholm conditions for the
edge calculus.

Our results can be extended to products of such manifolds, or to
manifolds that locally at infinity are products of such manifolds.

\subsection{Contents of the paper}
The paper consists of roughly two parts: the theoretical part and 
applications. The applications are included in the last
section, which we tried to write in such a way that they can,
to a large extent, be read independently of the rest of the paper. 
The reader interested only in applications can thus start 
immediately with Section \ref{sec.ex}. At least the main results 
of that section should be understandable {\em without any 
knowledge of groupoids.} 

We now describe in detail the main contents of the paper. We start with
reviewing the relevant topics related to groupoids and groupoid
$C^*$-algebras, so Section \ref{sec.LgLa} is mostly background
material on locally compact groupoids $\maG$, on their Haar systems,
and their $C^*$-algebras. We then review manifolds with corners and
tame submersions, and we recall the definitions of Lie groupoids in
the framework that we need, that is, that of manifolds with
corners. Note that all our Lie groupoids will be second countable and
Hausdorff. We finish this section with a review of some examples of
Lie groupoids that will play a role in what follows.

Section \ref{sec.exhaustive} contains preliminaries on exhaustive
families of representations from \cite{nistorPrudhon} and some general
results on groupoid $C\sp{\ast}$-algebras. We recall the notions of
strictly spectral and strictly norming families \cite{Roch} and the
main results from \cite{nistorPrudhon}. We then introduce groupoids
with \emph{ Exel's property}, respectively, \emph{strong Exel's
  property}, when the induced family of regular representations is
exhaustive for the reduced $C^*$-algebra, respectively, for the full
$C^*$-algebra. We prove that groupoids given by fibered pull-backs of
bundles of amenable Lie groups always have Exel's strong property, and
we introduce the class of \emph{\ssub\ groupoids}, given essentially
by glueing fibered pull-backs of bundles of Lie groups. Our main
result here is that \ssub\ Lie groupoids with amenable isotropy groups
always have Exel's strong property.

In Section \ref{sec4}, we define Fredholm Lie groupoids. Note that in
the following sections, we will work always in the setting of Lie
groupoids.  We provide a characterization of Fredholm Lie groupoids
using strictly spectral families of regular representations and show
that groupoids that have the strong Exel's property are Fredholm. It
follows that \ssub\ Lie groupoids with amenable isotropy groups are
Fredholm. We then specialize to algebras of pseudodifferential
operators on Lie groupoids and obtain the crucial Theorems
\ref{thm.nonclassical2} and \ref{thm.main.Fredholm}.

The last section of the paper, Section \ref{sec.ex}, contains examples 
and applications of our results, namely  of Theorem \ref{thm.main.Fredholm}. 
We start with group actions, define the associated transformation groupoid, 
and use this construction 
to describe Fredholm conditions related to several pseudodifferential
calculi: the $b$-groupoid that models operators on manifolds with
cylindrical and poly-cylindrical ends; the scattering groupoid that
models operators on manifolds that are asymptotically Euclidean, and
also operators on asymptotically hyperbolic spaces.  We then consider
the edge calculus and construct a suitable Fredholm \ssub \ groupoid
that will recover Fredholm conditions. This is a
particular case of a desingularization groupoid, whose construction is
outlined in the following subsection. We then give
more explicit examples of the desingularization and blow-up process. A
first example deals with the blow-up of a smooth submanifold along
another smooth, compact manifold.  The second example extends this
construction to manifolds with boundary.  Finally, the last example
deals with the iterated blow-up of a singular stratified subset of
dimension one. See the main body of the paper for specific references
to the existing literature.
  
For simplicity, in view of the applications considered in this paper,
we shall work almost exclusively with {\em Lie groupoids} (and their
reductions), although some definitions and results are valid in
greater generality. For the most part, our manifolds will be Hausdorff
and second countable.

\section{Groupoids and their $C^*$-algebras}
\label{sec.LgLa}
We recall in this section some basic definitions and properties of
groupoids.  Although our interest in applications lies mainly in Lie
groupoids, we have found it convenient to consider also the general
case of locally compact groupoids, so we shall discuss these two cases
in parallel.  We refer to Mackenzie's books \cite{MackenzieBook1,
  MackenzieBook2} for more details and, in general, for a nice
introduction to the subject of Lie groups and Lie groupoids, as well
as to further references and historical comments. See also
\cite{buneciSurvey, MoerdijkFolBook, renaultBook, WilliamsBook} for
the more specialized issues relating to analytic applications.  Most
of the needed results can be found also in \cite{nistorDesing}, whose
approach we also use here.

\subsection{Locally compact  groupoids }
\label{ssec.dLgLa}
Let us introduce groupoids as in
\cite{buneciSurvey,MackenzieBook2,renaultBook}.

\begin{definition} A {\em groupoid} is a pair $\maG = (\maG^{(1)}, \maG^{(0)})$,
where $\maG^{(i)}$, $i=0,1$, are sets, together with structural
morphisms
\begin{enumerate}[(1)]
 \item $d, r: \maG^{(1)} \to \maG^{(0)}$ (the ``domain'' and
   ``range''),
 \item $\iota : \maG^{(1)} \to \maG^{(1)}$ (the ``inverse''),
 \item $u : \maG^{(0)} \to \maG^{(1)}$ (the inclusion of units), and
 \item $\mu : \maG^{(2)} := \{(g, h) \in \maG^{(1)} \times \maG^{(1)}
   \vert\, d(g) = r(h) \} \to \maG^{(1)}$ (the product),
\end{enumerate}
with the following properties (we write $gh := \mu(g,h)$, for simplicity):
\begin{enumerate}[(1)]
 \item $d(gh) = d(h)$, $r(gh) = r(g)$ if $d(g) = r(h)$;
 \item $g_1(g_2 g_3) = (g_1 g_2) g_3$ for all $g_i \in \maG$ such that
   $d(g_i) = r(g_{i+1})$;
 \item $d(u(x)) = x = r(u(x))$, for all $x \in M$.
 \item $g u(d(g)) = g$ and $u(r(g)) g = g$ for all $g \in \maG$.
 \item $g \iota(g) = u(r(g))$ and $\iota(g) g = u(d(g))$ for all $g
   \in \maG$.
\end{enumerate}
\end{definition}

Moreover, we see that $\maG^{(0)}$ can be regarded as the set of
objects of a category with morphisms $\maG^{(1)}$. The objects of
$\maG$ will also be called {\em units} and the morphisms of $\maG$
will also be called {\em arrows}. Actually, {\em a groupoid} $\maG$ is
simply a small category in which every morphism is invertible.  (A
{\em small category} is one whose objects form a set.)  For
convenience, we shall identify $\maG = \maG\sp{(1)}$ and denote $M :=
\maG\sp{(0)}$. We shall write $\maG \tto M$ for a groupoid $\maG$ with
units $M$. {In this paper, $\maG$ always will denote a groupoid.}

\begin{definition}\label{def.lc.gr}
 A {\em locally compact groupoid} is a groupoid $\maG \tto M$ such
 that:
\begin{enumerate}[(1)]
  \item $\maG$ and $M$ are locally compact spaces, with $M$ Hausdorff;
  \item the structural morphisms $d, r, \iota, u,$ and $\mu$ are continuous;
  \item $d$ is surjective and open.
\end{enumerate}
\end{definition}

See Subsection \ref{ssec.ex.gr} for examples of groupoids.  Let us
notice that the definition requires only the space of units $M$ to be
Hausdorff. In this paper, however, all spaces will be assumed or
proved to be Hausdorff.

In what follows, we denote by $\maG_A := d\sp{-1}(A)$, $\maG^A=
r^{-1}(A)$ and $\maG_A\sp{B} := d\sp{-1}(A) \cap r\sp{-1}(B)$.  We
call $\maG_A\sp{A}$ the \emph{reduction} of $\maG$ to $A$. If $A$ is
\emph{$\maG$-invariant}, in the sense that $\maG_A\sp{A}=\maG_{A} =
\maG\sp{A} = r\sp{-1}(A)$, then $\maG_A$ is also a groupoid, {called
  the {\em restriction} of $\maG$ to $A$.} For any $x\in X$, $\maG_x^x
= d\sp{-1}(x) \cap r\sp{-1}(x)$ is a group, called the \emph{isotropy
  group} at $x$.

We will see several examples of groupoids in more detail when we
introduce Lie groupoids.

\subsection{Haar systems and $C\sp{\ast}$-algebras\label{ssec.Haar}}
We now recall the definition of a Haar system of a locally compact
groupoid and we use this opportunity to fix some more notations to be
used throughout the rest of the paper. We refer to \cite{buneciSurvey,
  ionescuWilliamsEHC, renaultBook} for more information on the topics
discussed in this subsection.

If $\maG$ is a locally compact groupoid, we shall denote by
$\maC_c(\maG)$ the space of continuous, complex valued, compactly
supported functions on $\maG$.

\begin{definition}
A {\em right Haar system} for a locally compact groupoid $\maG$ is a
family $\lambda =\{\lambda_x\}_{x\in M}$, where $\lambda_x$ is a Borel
regular measure on $\maG$ with support $\supp (\lambda_x) = d^{-1}(x)
=: \maG_{x}$ for every $x\in M = \maG\sp{(0)}$, satisfying
\begin{enumerate}[(i)]
 \item {\em The continuity condition:}
\begin{equation*}
 M \ni x \, \mapsto\, \lambda_{x}(\varphi) \, := \, \int_{\maG_x}\,
 \varphi(g) \de \lambda_x(g) \, \in \, \CC
\end{equation*}
is continuous for $\varphi\in\maC_c(\maG)$.
 \item  {\em The invariance condition:}
\begin{equation*}
 \int_{\maG_{r(g)}}\, \varphi(hg)\de\lambda_{r(g)}(h) \seq
 \int_{\maG_{d(g)}}\, \varphi(h)\de \lambda_{d(g)}(h)
\end{equation*}
for all $g\in\maG$ and $\varphi\in\maC_c(\maG)$.
\end{enumerate}
\end{definition}

We shall assume from now on that all our locally compact groupoids are
endowed with a (right) Haar system. Let thus $\maG$ be a locally
compact groupoid and $\{\lambda_x\}$ be the Haar system associated to
it. The space $\maC_c(\maG)$ has a natural product given by the
formula
\begin{equation*}
  (\varphi_1\ast \varphi_2)(g) \ede
  \int\limits_{d^{-1}(d(g))}\varphi_1(gh\sp{-1})
  \varphi_2(h)\de\lambda_{d(g)}(h) \,.
\end{equation*}
This makes $\maC_c(\maG)$ into an associative $*$-algebra with the
involution defined by
\begin{equation*}
 \varphi^*(g) \ede \overline{\varphi(g^{-1})}
\end{equation*}
for all $g\in\maG$ and $\varphi\in\maC_c(\maG)$. There also exists a
natural algebra norm on $\maC_c(\maG)$ defined by
\begin{equation*}
  \| f\|_I \ede \max \, \Bigl\{ \, \sup_{x\in M}\int\vert
  \varphi\vert\de\lambda_x, \, \sup_{x\in M}\int\vert
  \varphi^*\vert\de\lambda_x \, \Bigr\}.
\end{equation*}
The completion of $\maC_c(\maG)$ with respect to this norm $\|\cdot
\|_I$ will be denoted by $L^1(\maG)$.

Recall \cite{Dixmier} that a {\em $C\sp{\ast}$-algebra} is a complex
algebra $A$ together with a conjugate linear involution $*$ and a
complete norm $\| \ \|$ such that $(ab)^* = b^* a^*$, $\|ab\| \le
\|a\| \|b\|$, and $\|a^*a\| = \|a\|^2$, for all $a, b \in A$. Let
$\maH$ be a Hilbert space and denote by $\maL(\maH)$ the space of
linear, bounded operators on $\maH$. Then $\maL(\maH)$ is a
$C^*$-algebra.  A {\em representation} of a $C\sp{\ast}$-algebra $A$
on the Hilbert space $\maH_\pi$ is a $*$-morphism $\pi : A \to
\maL(\maH_\pi)$.

To a locally compact groupoid $\maG$ (endowed with a Haar system),
there are associated two basic $C^*$-algebras, the {\em full} and {\em
  reduced} $C^*$-algebras $\Cs{\maG}$ and $\rCs{\maG}$, whose
definition we now recall.

\begin{definition}\label{def.regular}
The (full) $C\sp{\ast}$-algebra associated to $\maG$, denoted
$C\sp{\ast}(\maG)$, is defined as the completion of $\maC_c(\maG)$
with respect to the norm
\begin{equation*}
  \| \varphi\| \ede \sup\limits_\pi\|\pi(\varphi)\| \,,
\end{equation*}
where $\pi$ ranges over all {\em contractive} $*$-representations of
$\maC_c(\maG)$.  Let us define as usual for any $x\in M$ the
\emph{regular} representation $\pi_x\, \colon\, C\sp{\ast}(\maG) \to
\maL(L^2(\maG_x,\lambda_x))$ by the formula
\begin{equation*}
  \pi_x(\varphi)\psi(g) \ede \varphi * \psi(g) \ede \int_{\maG_{d(g)}}
  \varphi(gh^{-1}) \psi(h) d\lambda_{d(g)}(h) \,, \quad \phi \in
  \maC_c(\maG) \,.
\end{equation*}
We then define similarly the reduced $C\sp{\ast}$-algebra
$C\sp{\ast}_{r}(\maG)$ as the completion of $\maC_c(\maG)$ with
respect to the norm
\begin{equation*}
  \| \varphi\|_r \ede \sup\limits_{x \in M}\|\pi_x(\varphi)\| \,.
\end{equation*}
The groupoid~$\maG$ is said to be \emph{metrically amenable} if the
canonical surjective $*$-homomorphism $\Cstar(\maG) \to
\Cstar_{r}(\maG)$, induced by the definitions above, is also
injective.
\end{definition}

Also, for further use, we note that if $\maG$ is second countable,
then $C^{\ast}(\maG)$ is a separable $C^{\ast}$-algebra.

\begin{remark}
\normalfont For any $\maG$-invariant, locally closed subset
$A\subseteq M$, the \emph{reduced} groupoid $\maG_A = \maG_A^A$ is
locally compact and has a Haar system $\lambda_A$ obtained by
restricting the Haar system $\lambda$ of $\maG$ to~$\maG_A$. In
particular, we can construct as above the corresponding
$C\sp{\ast}$-algebra $C\sp{\ast}(\maG_A)$ and the reduced
$C\sp{\ast}$-algebra $C\sp{\ast}_{r}(\maG_A)$.  For any closed subset
$A\subseteq M$, the subset $d^{-1}(A)\subseteq\maG$ is also closed, so
the restriction map $\maC_c(\maG)\to \maC_c(d^{-1}(A))$ is well
defined.  If $A$ is also $\maG$-invariant then the restriction extends
by continuity to both a $*$-homomorphism $\rho_A\colon
C\sp{\ast}(\maG)\to C\sp{\ast}(\maG_A)$ and a $*$-homomorphism
$(\rho_A)_{r}\colon C\sp{\ast}_{r}(\maG)\to C\sp{\ast}_{r}(\maG_A)$.
\end{remark}

We have the following well known, but important result
\cite{MRW87, MRW96, renault91} that we record for further reference.

\begin{proposition}\label{renault.exact}
Let $\maG \tto M$ be a second countable, locally compact groupoid with
a Haar system. Let $U \subset M$ be an open $\maG$-invariant subset,
$F := M \smallsetminus U$.
\begin{enumerate}[(i)]
\item\label{renault.exact_item1} $C\sp{\ast}(\maG_U)$ is a closed
  two-sided ideal of $C\sp{\ast}(\maG)$ that yields the short exact
  sequence
\begin{equation*}
  0\to C\sp{\ast}(\maG_U) \to
  C\sp{\ast}(\maG)\mathop{\longrightarrow}\limits^{\rho_F}
  C\sp{\ast}(\maG_{ F })\to 0 \,.
\end{equation*}

\item\label{renault.exact_item3} If $\maG_F$ is metrically
  a\-me\-nable, then one has the exact sequence
\begin{equation*}
  0\to C\sp{\ast}_{r}(\maG_U)\to C\sp{\ast}_{r}(\maG)
  \mathop{\xrightarrow{\hspace*{1cm}}}\limits^{(\rho_F)_{r}}
  C\sp{\ast}_{r}(\maG_F)\to 0 \,.
\end{equation*}

\item If the groupoids $\maG_F$ and $\maG_U$ (respectively, $\maG$)
  are metrically amenable, then $\maG$ (respectively, $\maG_U$) is
  also metrically amenable.
\end{enumerate}
\end{proposition}

\begin{proof}
The first assertion is well known, see for instance
\cite[Lemma~2.10]{MRW96}.  The second statement is in \cite[Remark
  4.10]{renault91}. The last part follows from (ii) and the Five
Lemma.
\end{proof}

\begin{remark}\label{rem.disjoint}
We notice that the exact sequence of Proposition \ref{renault.exact}
corresponds to a disjoint union decomposition $\maG = \maG_F \sqcup
\maG_U.$
\end{remark}

\subsection{Manifolds with corners and  Lie groupoids}\label{ssec.mc}
Even if one is interested only in analysis on smooth manifolds (that
is, in manifolds without corners or boundary), one important class of
applications will be to spaces obtained as the blow-up with respect to
suitable submanifolds, which leads to manifolds with corners, as each
blow-up increases the highest codimension of corners by one. To model
the analysis on these spaces, we need Lie groupoids. The advantage of
using Lie groupoids is that they have a distinguished class of Haar
systems defined using the smooth structure.

We begin with some background material following \cite{nistorDesing}.
It will be important to distinguish between smooth manifolds without
corners (or boundaries) and (smooth) manifolds with corners. The
former will be {\em smooth manifolds}, while the later will be simply
{\em manifolds}. Thus, in this paper, a {\em manifold} $M$ is a {\em
  second countable, Hausdorff} topological space locally modeled by
open subsets of $[-1, 1]\sp{n}$ with smooth coordinate changes. In
particular, our manifolds may have corners. By contrast, a {\em smooth
  manifold} will not have corners (or boundary).  A point $p \in M$ of
a manifold (with corners) is called of {\em depth} $k$ if it has a
neighborhood $V_p$ diffeomorphic to $[0, a)^{k} \times (-a, a)^{n-k}$,
  $a > 0$, by a diffeomorphism $\phi_p : V_p \to [0, a)^{k} \times
    (-a, a)^{n-k}$ with $\phi_p(p) = 0$.

The set of {\em inward pointing tangent vectors} in $v \in T_{x}(M)$
define a closed cone denoted $T_{x}\sp{+}(M)$.  A function $f : M \to
M_1$ between two manifolds with corners will be called {\em smooth} if
its components are smooth in all coordinate charts. A little bit of
extra care is needed here in defining the derivatives. This is
illustrated clearly in one dimension: if $f : [0, 1] \to \RR$, then
$f'(0) := \lim_{h \to 0, t>0} h\sp{-1} (f (h) - f(0))$, whereas $f'(1)
:= \lim_{h \to 0, t>0} h\sp{-1} (f (1) - f(1-h))$. Since $\lim_{h \to
  0} h\sp{-1} (f (x + h) - f(x)) = \lim_{h \to 0} h\sp{-1} (f (x) -
f(x-h)) =: f'(x)$, for $0 < x < 1$, this defines unambiguously
$f'(x)$. A similar comment is in order in higher dimensions as
well. In local coordinates, a smooth function $[0, 1]^n \to [0, 1]^k$
is one that is the restriction of a smooth function $\RR^n \to \RR^k$.

Let $M$ and $M_1$ be manifolds with corners and $f : M_1 \to M$ be a
smooth map. Then $f$ induces a vector bundle map $df : TM_1 \to TM$
such that $df(T_z\sp{+}(M_1)) \subset T_{f(z)}\sp{+}M$. If the smooth
map $f : M_1 \to M$ is injective, has injective differential $df$, and
has {\em locally closed range,} then we say that $f(M_1)$ is a {\em
  submanifold} of $M$. We stress the condition that submanifolds be
locally closed (that is, the intersection of an open and a closed
subset). Other than this conditions, our concept of submanifold is the
most general possible.

\begin{definition}\label{def.tame.submersion}
A {\em tame submersion} $h$ between two manifolds with corners $M_1$
and $M$ is a smooth map $h : M_1 \to M$ such that its differential
$dh$ is surjective everywhere and
\begin{equation*}
 (dh_x)\sp{-1} (T_{h(x)}\sp{+}M) \, = \, T_{x}\sp{+}M_1 \,.
\end{equation*}
(That is, $dh(v)$ is an inward pointing vector of $M$ if, and only if,
$v$ is an inward pointing vector of $M_1$.)
\end{definition}

Clearly, if $h : M_1 \to M$ is a tame submersion of manifolds with
corners, then $x$ and $h(x)$ will have the same depth. We have the
following well known lemma (see for instance \cite{nistorDesing}).

\begin{lemma}\label{lemma.corners}
Let $h : M_1 \to M$ be a tame submersion of manifolds with corners.
\begin{enumerate}[(i)]
 \item For $m_1 \in M_1$, there exists an open neighborhood $U$ of
   $m_1$ in $M_1$ such that $h(U)$ is open and the restriction of $h$
   to $U$ is a $\CI$ fibration with basis $h(U)$.
 \item Let $L \subset M$ be a submanifold, then $L_1 := h\sp{-1}(L)$
   is a submanifold of $M_1$ of rank $\le$ the rank of $L$.
\end{enumerate}
\end{lemma}

We now define Lie groupoids roughly by replacing our spaces with
manifolds with corners and the continuity conditions with smoothness
conditions. The reason why Lie groupoids play an important role in
{applications} is that they model the distribution kernels of many
interesting classes of pseudodifferential operators.

\begin{definition}\label{def.Lie.gr}
A {\em Lie groupoid} is a groupoid $\maG \tto M$ such that
\begin{enumerate}[(i)]
  \item $\maG$ and $M$ are manifolds with corners,
  \item the structural morphisms $d, r, \iota,$ and $u$ are smooth,
  \item $d$ is a tame submersion of manifolds with corners, and
  \item $\mu$ is smooth.
\end{enumerate}
\end{definition}

In particular, the Lie groupoids used in this paper are locally
compact groupoids and are second countable and Hausdorff.  In
different contexts in other papers, it is sometimes useful {\em not to
  assume} the set of arrows $\maG$ of a Lie groupoid to be Hausdorff,
although the space of units $M$ of a Lie groupoid $\maG$ is always
assumed to be Hausdorff. Lie groupoids were introduced by Ehresmann. A
more general class also useful in applications is that of
\emph{continuous family groupoids,} where we assume continuity only
along the units, keeping smoothness along the fibers. See \cite{LMN1,
  PatersonBook}.

For the rest of this section, $\maG$ will always be a {\em Lie
  groupoid}.  In fact, most of the groupoids that we will consider in
what follows will be Lie groupoids. However, one has to be careful
since the restriction of a Lie groupoid is not necessarily a Lie
groupoid itself. Moreover, we shall assume that all our Lie groupoids
are Hausdorff.

Since $d$ and $r$ are tame submersions, it follows from Lemma
\ref{lemma.corners} that the fibers $\maG_x := d\sp{-1}(x)$, $x \in
M$, are smooth manifolds (that is, they have no corners). Similarly,
the same lemma implies that the set $\maG^{(2)} \subset \maG \times
\maG$ of composable units is a manifold as well (but it may have
corners). In particular, it makes sense to consider smooth maps
$\maG^{(2)} \to \maG^{(1)}$.

\begin{remark}\label{rmk.liealgebroid}
Let $\maG \tto M$ be a Lie groupoid and let $A(\maG) := \cup_{x \in M}
T_x \maG_x$, where $\maG_x := d^{-1}(x)$, as usual. Let us denote by
$f_* : TM_1 \to TM_2$ the differential of a smooth map $f : M_1 \to
M_2$. Then $A(\maG)$ is the restriction to units of the vector bundle
$\ker (d_*)$, and hence it has a natural structure of vector bundle
over $M$. It is called the {\em Lie algebroid} of $\maG$. The
differential $r_*$ of the range map $r : \maG \to M$ then induces a
map $\varrho := r_*\vert_{A(\maG)} \to TM$, called the {\em anchor
  map} of $A(\maG)$. We let $\Lie(\maG) := \varrho(\Gamma(M;
A(\maG)))$ denote the image by $\varrho$ of the space of sections of
$A(\maG)$.  It is coincides with the image by $r_*$ of the space
$d$-vertical (i.e. tangent to the fibers $\maG_x$ of $d$) vector
fields on $\maG$ that are invariant for the action of $\maG$ on itself
by right multiplication.  Therefore both $\Gamma(M; A(\maG))$ and
$\Lie(\maG) := \varrho(\Gamma(M; A(\maG)))$ are Lie algebras for the
Lie bracket. Both Lie algebras will play an important role in the
analysis of differential operators.
\end{remark}

\begin{remark}\label{rmk.lie.haar}
We can use the vector bundle $A(\maG)$ to define a Haar system. Let $D
:= |\Lambda\sp{n}A(\maG)|$, where $n$ is the dimension of the Lie
algebroid of $\maG$. The pull-back vector bundle $r\sp{*}(D)$ is the
bundle of 1-densities along the fibers of $d$. A trivialization of $D$
will hence give rise to a right invariant set of measures on $\maG_x$
and hence to a right Haar system.
\end{remark}

We conclude that a Lie groupoid always has a (right) Haar system, and
we shall always assume that it is obtained as in the remark above.

\subsection{Examples of Lie groupoids}\label{ssec.ex.gr}
We continue with various examples of constructions of Lie groupoids
that will be needed in what follows. Recall that we only consider
Hausdorff groupoids in this paper. Also, recall that a Lie groupoid
$\maG$ is metrically amenable if the canonical morphism $C^\star(\maG)
\to C^\star_r(\maG)$ is an isomorphism.

\begin{example}
\label{ex.Lie.group}
Any Lie group $G$ gives rise to a Lie groupoid with set of units
reduced to one point: the identity element of $G$. Thus $d, r : G \to
M:=\{e\}$ are constant, $\mu : G \times G \to G$ is the group
multiplication, $\iota(g) := g^{-1}$ is the (usual) inverse, and $u :
M \to G$ is the inclusion of the unit.  We have that the Lie algebroid
$A(G)$ of $G$ coincides with the Lie algebra of $G$ (the definition of
a Lie algebroid was recalled in Remark~\ref{rmk.liealgebroid}).  Hence
$\Lie(G):=\Gamma(A(\maG))$ is also the Lie algebra of $G$.  We have
that $G$ is metrically amenable if, and onl if, it is amenable in the
usual sense of groups. In particular, if $G$ is solvable, then $G$ is
amenable.
\end{example}

\begin{example}\label{ex.space}
Let $M$ be a manifold with corners (hence {\em Hausdorff} by our
conventions), let $\maG = M$, and $d = r = u = \iota = id_M$. Then
$\maG$ is a Lie groupoid with only units. We shall call a Lie groupoid
with these properties a {\em manifold}. We have that $A(M) = M \times
\{0\}$, that is, the {\em zero} vector bundle. We have $\Lie(\maG) =
0$.
\end{example}

\begin{example} \label{ex.product}
Let $\maG_i \to M_i$, $i = 1, 2$, be two Lie groupoids. Then $\maG_1
\times \maG_2$ is a Lie groupoid with units $M_1 \times M_2$ and
$A(\maG_1 \times \maG_2) = A(\maG_1) \times A(\maG_2)$. We have that
$\Lie(\maG_1 \times \maG_2)$ is a suitable topological tensor product
of $\Lie(\maG_1)$ and $\Lie(\maG_2)$.
\end{example}

The product of a manifold with a Lie group is thus again a Lie
groupoid. We can ``twist'' this example to obtain a (smooth) ``bundle
of Lie groups.''

\begin{example}
\label{ex.BLG}
Let $G$ be a Lie group with automorphism group $\Aut(G)$ and let $P
\to M$ be a smooth, locally trivial, principal $\Aut(G)$-bundle, with
$M$ and $P$ manifolds with corners (hence Hausdorff, by our
conventions).  Then the associated fiber bundle $\maG := P
\times_{\Aut(G)} G$ with fiber $G$ is a Lie groupoid with units $M$. A
Lie groupoid of this form will be called a {\em smooth bundle of Lie
  groups}. It satisfies $d = r$. If $P \to M$ is also differentiable
(with $P$ and $M$ manifolds, possibly with corners), then $\maG$ is a
Lie groupoid. We have $A(\maG) \simeq P \times_{\Aut(G)} Lie(G)$ is a
smooth bundle of Lie algebras and $\varrho : A(\maG) \to TM$ is the
zero map. The Lie algebra $\Lie(\maG)$ identifies thus with the family
of sections of the bundle of Lie algebras $P \times_{\Aut(G)}
Lie(G)$. If $G$ is amenable, then $\maG$ is metrically amenable. The
fact that $\maG$ is metrically amenable when $G$ is amenable will be
of crucial importance to us and will be used when the group $G$ is a
solvable Lie group.
\end{example}

The following simple example of the ``pair groupoid'' will be
fundamental in what follows.

\begin{example}
\label{ex.pair}
Let $M$ be a {\em smooth} manifold. (Thus $M$ has no corners,
according to our terminology). Then the {\em pair groupoid} of $M$ is
$\maG := M \times M$. It is a groupoid with units $M$. The structural
morphisms are as follows: $d$ is the second projection, $r$ is the
first projection, and the product $\mu$ is given by $(m_1, m_2)(m_2,
m_3) = (m_1, m_3)$. This determines $u(m) = (m, m)$ and $\iota(m,m') =
(m', m)$.  We thus have that the product on $\CI_c(\maG) = \CI_c(M
\times M)$ is the product of operators with integral kernels.  We have
$A(M \times M) = TM$, with anchor map the identity map.  One crucial
feature of the pair groupoid is that, for any $x$, the regular
representation $\pi_x$ defines an isomorphism between
$C\sp{\ast}(M\times M)$ and the ideal of compact operators in
$\maL(L^2(M))$. In particular, all pair groupoids are metrically
amenable.
\end{example}

\begin{example}
Let $M$ be a smooth manifold (so without corners) and $\tilde M$ be it
universal covering. Let $\pi_1(M)$ be the fundamental group of $M$
associated to some fixed point of $M$.  A related example to the pair
groupoid is that of the {\em path groupoid} $\maP(M) := (\tilde M
\times \tilde M)/\pi_1(M)$ of $M$, which will have the same Lie
algebroid as the pair groupoid: $A(\maP(M)) = TM$.  The analysis
associated to $\maP(M)$ is that of $\pi_1(M)$-{\em invariant}
operators on the covering space $\tilde M$, thus quite different to
that of the pair groupoid. This underscores the importance of choosing
the right integrating groupoid when interested in Analysis.
\end{example}

We extend the example of the pair groupoid by defining {\em fibered
  pull-back groupoids} \cite{HigginsMackenzie1, HigginsMackenzie2} (we
use the terminology in \cite{nistorDesing}, however).

\begin{example}\label{ex.pullback} Let again $M$ and $L$ be
manifolds and $f : M \to L$ be {\em continuous} map. Let $\maH \tto L$
be a Lie groupoid. The {\em fibered pull-back groupoid} is then
\begin{equation*}
  f\pullback (\maH) \ede \{\, (m, g, m') \in M \times \maH \times M,
  f(m) = r(g),\, d(g) = f(m') \, \} \,,
\end{equation*}
with units $M$ and product $(m, g, m') (m', g', m'') = (m, g
g', m'')$. We shall also sometimes write $M \times_f \maH \times_f
M = f\pullback (\maH)$ for the fibered pull-back groupoid. Assume now
that $f$ is a tame submersion (in particular, that it is smooth). Then
$ f\pullback (\maH)$ is a Lie groupoid and its Lie algebroid is
\begin{equation*}
 A( f\pullback (\maH)) \seq \{ (\xi , X)\vert\ \xi \in A(\maH)\,,\ X \in TM\,,\
 \varrho(\xi) = f_*(X) \}\,,
\end{equation*}
the {\em thick pull-back} Lie algebroid $f\pullback A(\maH)$ of
$A(\maH)$, see \cite{MackenzieBook1, MackenzieBook2}.
\end{example}

The following particular example of a fibered pull-back will be useful
for the study of the $b$-groupoid $\maG_b$ in Section \ref{sec.ex}.

\begin{example}\label{ex.help-for-b}
Let us assume that $M$ is a smooth manifold and let $B$ be its set of
connected components. Let $G$ be a Lie group. We let $\maH := B \times
G$, the product of a manifold and a Lie group, and $f : M \to B$ be
the map that associates to a point its connected component. Then
$f\pullback(\maH)$ is the topological disjoint union of the groupoids
$(F \times F) \times G$ (product of the pair groupoid and a Lie group)
for $F$ ranging through the connected components of $M$.
\end{example}

\section{Exhaustive families of representations and Exel's question}
\label{sec.exhaustive}

We next recall some basic facts on {\em exhaustive} families of
representations following \cite{nistorPrudhon}. A groupoid has {\em
  Exel's property} when its set of regular representations is
exhaustive. Exel \cite{ExelInvGr} has asked which groupoids have this
property.  Establishing that a groupoid has Exel's property will turn
out to be important in establishing Fredholm conditions in the next
section. At the end of the section, we prove that groupoids given by
fibered pull-backs of bundles of Lie groups always have Exel's
property, which leads us to introduce, and extend this property to,
the more general class of \emph{\ssub\ groupoids}.

\subsection{Exhaustive families of representations
of $C\sp{\ast}$-algebras}\label{subsect-exhaustive} A two-sided ideal
$I \subset A$ of a $C^*$-algebra $A$ is called {\em primitive} if it
is the kernel of an irreducible representation of $A$. We shall denote
by $\Prim(A)$ the set of primitive ideals of $A$. The trivial ideal
$A$ (of $A$) is {\em not} considered a primitive ideal (of $A$), so $A
\notin \Prim(A)$.  For any representation $\phi$ of $A$, we define its
{\em support} $\supp(\phi) \subset \Prim(A)$ as the set of primitive
ideals of $A$ {\em containing} $\ker(\phi)$. Let $J$ be a closed,
two-sided ideal of $A$.  Then $\Prim(J) = \{ I \in \Prim(A)\vert\, J
\not\subset I\}$ and $\Prim(A/J)$ identifies with $\{ I \in
\Prim(A)\vert\, J \subset I\} = \Prim(J)^c$.  We endow $\Prim(A)$ with
the hull-kernel topology, which is the topology whose open sets are
those of the form $\Prim(J)$, with $J$ a two-sided ideal of $A$.

We shall need the following definition \cite{nistorPrudhon}.

\begin{definition}\label{def.exhaustive}
Let $\maF$ be a set of representations of a $C\sp{\ast}$-algebra $A$.
We say that $\maF$ is {\em exhaustive} if $\Prim( A ) = \bigcup_{\phi
  \in \maF}\, \supp(\phi)$.
\end{definition}

In practice, we rather need families with the following definition
\cite{Roch}.

\begin{definition}\label{def.ip}
Let $\maF$ be a set of representations of a unital
$C\sp{\ast}$-algebra $A$.
\begin{enumerate}[(i)]
\item We say that $\maF$ is {\em strictly norming} if, for any $a \in
  A$, there exists $\phi \in \maF$ such that $\|\phi(a)\| = \|a\|$.
\item We say that $\maF$ is {\em strictly spectral} if, for any $a \in
  A$, we have that $a$ is invertible in $A$ if, and only if, $\phi(a)$
  is invertible for all $\phi \in \maF$.
\end{enumerate}
\end{definition}

The two definitions above can be formulated for sets of morphisms or
sets of primitive ideals. We also have that $\maF_*:= \{\pi \in
\maF\vert\, \pi \neq 0\}$ has the same properties (``exhaustive,''
``strictly norming,'' ... ) as $\maF$.

If $A$ is non-unital, we modify the last definition as follows
\cite{nistorPrudhon}.  Let $A\sp{+} := A \oplus \CC$ and $\chi_0
\colon A\sp{+} \to \CC$ be the morphism defined by $\chi_0 = 0$ on $A$
and $\chi_0(1) = 1$.  We then replace $A$ with $A\sp{+}$ and $\maF$
with $\maF^{+} := \maF \cup \{\chi_0\}$.  This works also for
exhaustive families since $\maF$ is exhaustive for $A$ if, and only
if, $\maF^{+}$ is exhaustive for $A^{+}$. Sometimes it is convenient
to use the following alternative characterization of strictly spectral
sets of representations. The set of representations $\maF$ of $A$ is
strictly spectral if it satisfies the following property: $1 + a \in
A\sp{+}$, $a \in A$, is invertible if, and only if, $1 + \phi(a)$ is
invertible for any $\phi \in \maF$.

The set of all equivalence classes of irreducible representations of a
unital $C\sp{\ast}$-algebra is strictly norming (see \cite{Dixmier,
  ExelInvGr}).  Therefore any exhaustive family is strictly norming
\cite{nistorPrudhon}. We recall the following results from
\cite{nistorPrudhon, Roch}, which establish the relations between
these notions.

\begin{theorem}\label{thm.exhaustive}
Let $\maF$ be a set of non-degenerate representations of a
$C\sp{\ast}$-algebra $A$.  Then $\maF$ is strictly norming if, and
only if, it is strictly spectral.  Every exhaustive set of
representations is strictly spectral.  If $A$ is furthermore
separable, then the converse is also true.
\end{theorem}

Let $A$ be a $C\sp{\ast}$-algebra and $I \subset A$ be a closed
two-sided ideal.  Recall that any nondegenerate representation $\pi
\colon I \to \maL(\maH)$ extends to a unique representation $\pi
\colon A \to \maL(\maH)$.  (See \cite[Proposition~2.10.4]{Dixmier}.)
This leads to the following results (see \cite{nistorPrudhon},
Proposition 3.15 and Corollary 3.16).

\begin{proposition}\label{prop.ideal}
Let $I \subset A$ be an ideal of a $C\sp{\ast}$-algebra. Let $\maF_I$
be a set of nondegenerate representations of $I$ and $\maF_{A/I}$ be a
set of representations of $A/I$.  Let $\maF := \maF_I \cup
\maF_{A/I}$, regarded as a family of representations of $A$. If
$\maF_I$ and $\maF_{A/I}$ are both exhaustive, then $\maF$ is also
exhaustive. The same result holds by replacing exhaustive with
strictly norming.
\end{proposition}

In the following corollary, one should think of the invertibility of
$a$ in $A/I$ as an ``ellipticity'' condition.

\begin{corollary}\label{cor.ideal}
Let $I \subset A$ be an ideal of a unital $C\sp{\ast}$-algebra $A$ and
let $\maF_I$ be a strictly spectral set of nondegenerate
representations of $I$. Let $a \in A$. Then $a$ is invertible in $A$
if, and only if, it is invertible in $A/I$ and $\phi(a)$ is invertible
for all $\phi \in \maF_I$.
\end{corollary}

We now address Morita equivalence {\cite{rieffelInducedCstar} in a
  very simple form (see also \cite{PatersonBook, renaultBook}).}  It
is known that there exists a homeomorphism between the primitive ideal
spectra of Morita equivalent $C^*$-algebras, so exhaustive families of
representations will correspond to exhaustive families of
representations under Morita equivalence.  In particular, we have the
following result.

\begin{proposition}\label{prop.Morita} Let $\maF$ be
a set of representations of a $C^{*}$-algebra $A$.
\begin{enumerate}[(i)]
 \item If $\maF_n := \{\pi \otimes 1\vert \, \pi \in \maF\}$ is the
   corresponding family of representations of $M_n(A) = A \otimes
   M_n(\CC)$, then $\maF$ is exhaustive if, and only if, $\maF_n$ is
   exhaustive.
 \item Let $I \subset A$ be a closed, two-sided ideal.  If $\maF$ is
   exhaustive, then $\maF_I := \{\pi\vert_I \vert\, \pi \in \maF\}$ is
   an exhaustive family of representations of $I$.
\item Let $e \in M_n(A)$ be a projection. Then $\maF_n$ defines, by
  restriction, an exhaustive family of representations of $eM_n(A)e$.
\end{enumerate}

\end{proposition}

\begin{proof}
The first part follows by Morita equivalence. The second part follows
since $\supp(\phi\vert_I) = \supp(\phi) \cap \Prim(I)$. The last part
is proved by combining the first two parts. Indeed, $\maF_n$ defines
an exhaustive set of representations of $M_n(A)$, and hence an
exhaustive set of representations of $M_n(A)eM_n(A)$. Since the
algebra $eM_n(A)e$ is Morita equivalent to $M_n(A)eM_n(A)$, by
Proposition 4.27 in \cite{GBVF} (see also \cite{rieffelInducedCstar}),
the result follows.
\end{proof}

\subsection{Exel's question and properties}
We will let $\maG$ denote a Lie groupoid and let $\maR(\maG)$ will
denote its set of regular representations.  The discussion below
usually makes sense in the more general setting of locally compact
groupoids, but we restrict, nevertheless to Lie groupoids, for
simplicity.

In \cite{ExelInvGr}, Exel has asked for which groupoids $\maG$ is the
set $\maR(\maG)$ a strictly norming set of representations of
$\rCs{\maG}$. An example due to Voiculescu, see \cite{nistorPrudhon},
shows that not every locally compact groupoid has this property.  (The
question whether all Lie groupoids have Exel's property is still
open.) It is easier to work with exhaustive families of
representations, so we introduce the following definition.

\begin{definition}
We say that $\maG$ has {\em Exel's property} if $\maR(\maG)$, the set
of regular representations of $\maG$, is an exhaustive set of
representations for $\rCs{\maG}$. We say that $\maG$ has {\em Exel's
  strong property} if $\maR(\maG)$ is an exhaustive set of
representations for $\Cs{\maG}$.
\end{definition}

We have that
\begin{equation*}
 \bigcup_{\phi \in \maF}\, \supp(\phi) = \bigcup_{\phi \in \maF}\,
 \Prim(A)\backslash \Prim(\ker(\phi)) = \Prim(A) \backslash
 \bigcap_{\phi \in \maF}\, \Prim(\ker(\phi)).
\end{equation*}
Hence, if $\maF$ is an exhaustive family of representations of a
$C^*$-algebra $A$, then $\bigcap_{\phi \in \maF}\, \Prim(\ker(\phi))=
\emptyset,$ and hence $\maF$ is a faithful family of representations
of $A$.  In particular, $\maG$ has Exel's strong property if, and only
if, $\maG$ has Exel's property and is metrically amenable.

We shall need the following consequence of Propositions
\ref{renault.exact} and \ref{prop.ideal} of the previous subsection.

\begin{corollary}\label{cor.sufficient}
Let $U \subset M$ be an open, $\maG$-invariant subset and assume
$\maG$ to be second countable. We let $F := M \smallsetminus U$. Let
us assume that $\maG_U$ and $\maG_F$ have Exel's strong property. Then
$\maG$ also has Exel's strong property.
\end{corollary}

\begin{proof}
Let $I := \Cs{\maG_U} = \rCs{\maG_U}$, $A := \Cs{\maG}$. Then $A/I
\simeq \Cs{\maG_F} \simeq \rCs{\maG_F}$ by Proposition
\ref{renault.exact}(ii).  Since $\maG_F$ and $\maG_U$ are metrically
amenable, it follows that $\maG$ is metrically amenable as well, by
Proposition \ref{renault.exact}(iii).  Proposition \ref{prop.ideal}
then gives that $\maR(\maG) = \maR(\maG_U) \cup \maR(\maG_F)$ is an
exhaustive family of representations of $\Cs{\maG} \simeq \rCs{\maG}$,
that is, $\maG$ has Exel's strong property, as claimed.
\end{proof}

More generally, we can consider groupoids given by a filtration.

\begin{corollary}\label{cor.sufficient.filtration}
Let $\maG\tto M$ be a second countable groupoid and assume that $U_i$
are open, $\maG$-invariant subsets of $M$, such that $ \emptyset \ede
U_{-1} \subset U_0 \subset \ldots \subset U_{i-1} \subset U_i \subset
\ldots \subset U_N := M \,, $ and for each $S := U_{i} \smallsetminus
U_{i-1}$, $\maG_S$ has Exel's strong property. Then $\maG$ also has
Exel's strong property.
\end{corollary}

\begin{proof}
Corollary \ref{cor.sufficient} yields the result by induction, as
follows. If $N=1$, it reduces to Corollary \ref{cor.sufficient}.
Assuming that the result holds for a filtration in $N-1$ sets, then
$\maG_{U_{N-1}}$ has Exel's strong property. Since by assumption
$\maG_{M\setminus U_{N-1}}$ also has, the result follows again by
Corollary \ref{cor.sufficient}.
\end{proof}

Here is a simple, but very important example of a class of Lie
groupoids $\maG$ for which the set $\maR(\maG)$ of regular
representations is an exhaustive set of representations of
$\Cs{\maG}$.  The idea of the proof below is that, locally, the
$C^*$-algebra of $\maG$ is of the form $C^*(G) \otimes \maK$, where
$\maK$ are the compact operators on the (typical) fiber of $f$. We
found it convenient to formalize the proof of this well-known result
using Morita equivalence of groupoids.

\begin{proposition} \label{prop.exhaustive}
Let $\maH \tto L$ be a bundle of Lie groups with fiber $G$ and let $f
: M \to L$ be a tame submersion. Define $\maG := f\pullback
(\maH)$. Then $\maG$ is a Lie groupoid with Exel's property. If $G$ is
amenable, then $\maG$ has Exel's strong property.
\end{proposition}

\begin{proof}
We have already discussed that $\maG := f\pullback (\maH)$ is a Lie
groupoid (because $f$ is a tame submersion). We endow $\maG$ with one
of the standard Haar measures coming from the Lie groupoid structure.
Let
\begin{equation*}
  X \seq M \times_L \maH \ede \{ (m, \gamma) \in M \times \maH\vert
  f(m) = r(\gamma)\}\,.
\end{equation*}
Then $X$ defines a Morita equivalence between $\maG$ and $\maH$, and
hence a Morita equivalence between their full and reduced
$C^*$-algebras of these groupoids, by the classical results of
\cite{MRW87,SW12}. See also \cite{Tu04}.  In particular,
$\Prim(\Cs{\maG})$ and $\Prim(\Cs{\maH})$ are homeomorphic by natural
homeomorphisms \cite{rieffelInducedCstar}. The same is true for
$\Prim(\rCs{\maG})$ and $\Prim(\rCs{\maH})$.

All irreducible representations of $\rCs{\maH}$ factor through an
evaluation morphism $e_x : \rCs{\maH} \to \rCs{G_x} \simeq \rCs{G}$,
for some $x \in M$, where $G_x$ is the fiber of $\maH \to M$ above
$x$. We have that the regular representation $\pi_x$ is obtained from
the regular representation of $G_x$ via $e_x$, the evaluation at
$x$. Hence $\maG$ has Exel's property.

Let us assume now that $G$ is amenable. Hence the fibers of $\maH \to
M$ are amenable, and therefore $\maH$ is metrically amenable. Hence
$\maG$ is also metrically amenable. Since $\maG$ was already proved to
have Exel's property, it follows that it has also Exel's strong
property.
\end{proof}

We notice that if $\maG$ is (isomorphic to a groupoid) as in the
proposition above, then $\maH \tto L$ and $f : M \to L$ are uniquely
determined, up to equivalence.  Indeed, $L$ is diffeomorphic to the
set of orbits of $\maG$ (acting on $M$), $f$ becomes the quotient map
$M \to M/\maG$ and $\maH$ is obtained from the isotropy groupoid of
$\maG$.

The above results suggest the following definition.

\begin{definition}\label{def.bfb}
Let $\maH$ be a groupoid with units $M$. We say that $\maH$ is a {\em
  \ssub\ groupoid} with filtration $(U_i)$ if $U_i$ are open,
$\maH$-invariant subsets of $M$,
 \begin{equation*}
  \emptyset \ede U_{-1} \subset U_0 \subset \ldots \subset U_{i-1}
  \subset U_i \subset \ldots \subset U_N := M \,,
 \end{equation*}
and each $S := U_{i} \smallsetminus U_{i-1}$ is a manifold, possibly
with corners, such that there exist a Lie group bundle $G_S \to B_S$
and a tame submersion $f_S : S \to B_S$ of manifolds with corners such
that
\begin{equation*}
  \maH_S\, \simeq \, f_S \pullback (G_S) \,.
\end{equation*}
A groupoid $\maH$ will be called a {\em \ssub} groupoid if it is a
{\em \ssub} groupoid for some filtration.
\end{definition}

\begin{remark}\label{rem.restriction}
It is clear from the definition that if $\maH$ is a \ssub\ groupoid,
then, with the notation of the definition, each $\maH_{U_k
  \smallsetminus U_j}$, $j < k$, is also a \ssub\ groupoid. We have to
notice, however, that $\maH_{U_k \smallsetminus U_j}$ may not be a Lie
groupoid, even if $\maH$ is a Lie groupoid, because $U_k
\smallsetminus U_j$ may not be a manifold (even if it is a union of
manifolds).
\end{remark}

As we shall see, the class of \ssub\ groupoids is large enough to
comprise many classes of groupoids and operators that arise in
practice. Combining Corollary \ref{cor.sufficient.filtration} with
Proposition \ref{prop.exhaustive}, we now obtain the following result.

\begin{theorem}\label{thm.lemma.Exel}
Let $\maG \tto M$ be a \ssub\ Lie groupoid and assume that all the
isotropy groups $\maG_x^x$, $x \in M$, are amenable. Then the set
$\maR(\maG)$ of regular representations of $\maG$ is an exhaustive
family of representations of $\Cs{\maG}$.  In particular, $\maG$ is
metrically amenable and has Exel's property.
\end{theorem}

\begin{proof}
Let us use the notation of Definition \ref{def.bfb}.  Indeed then, the
isotropy group $\maG_x^x$ is the fiber of $G_S \to B_S$ above
$f_S(x)$, for $x \in S$. Thus, by Proposition \ref{prop.exhaustive},
all the groupoids $\maG_S$ are metrically amenable and have Exel's
property.  Using Corollary \ref{cor.sufficient.filtration}, we obtain
that $\maG$ has Exel's strong property, that is, that it is metrically
amenable and has Exel's property.
\end{proof}

In the next section, we will see how this property plays a role in
establishing Fredholm conditions for operators on groupoids.

\begin{remark}
Let $\maH$ be a \ssub\ Lie groupoid as in Definition \ref{def.bfb}.
In view of the remark \ref{rem.disjoint}, we have then the following
disjoint union decomposition:
\begin{equation*}
 \maH \seq \sqcup_{S}\, f_S \pullback (G_S) \seq \sqcup_{S}\, S
 \times_B S \times_B G_S\,,
\end{equation*}
where $\times_B$ is the fibered product over the base (that is, over
$B_S$).  We notice also that two regular representations $\pi_x$ and
$\pi_y$ are unitarily equivalent if there is $g \in \maG$ such that
$d(g) = x$ and $r(g) = y$. For a \ssub\ groupoid, this is then the
case exactly if $x, y \in S:=U_{i-1} \smallsetminus U_i$ {\em and}
$f_S(x) = f_S(y) = b \in B_s$, in which case, these two
representations act, up to a unitary equivalence, on the set of square
integrable functions on the space $Z_b:=f_S^{-1}(b) \times
(G_S)_b$. We denote by $\pi_b$ the corresponding representation on
$L^2(Z_b)$, $b \in B_S$.
\end{remark}

\section{Fredholm conditions}\label{sec4}
\label{sec.fredholm}

We now study Fredholm conditions for operators in algebras
$\mathbfPsi$ containing a reduced groupoid $C\sp{\ast}$-algebra
$C\sp{\ast}_r(\maG)$ as an essential ideal. (Recall that an ideal of
an algebra is {\em essential} if its annihilator vanishes.) The
groupoids for which we obtain the kind of Fredholm conditions that we
want (the kind that are typically used in practice) will be called
``Fredholm groupoids.'' They are introduced and discussed next.
Examples of Fredholm groupoids will be provided in the next section.

\emph{Throughout the rest of this paper, $\maG \tto M$ will denote a
  second countable Lie groupoid, which is hence a locally compact
  groupoid and the choice of a metric on $A(\maG)$ gives rise to a
  Haar system $\lambda_x$. Also, recall that all our Lie groupoids are
  assumed to be Hausdorff.}

\subsection{Fredholm groupoids and their characterization}
We now introduce Fredholm groupoids and give a first characterization
of these groupoids.  Recall that $\pi_x \colon C^{\ast}(\maG) \to
\maL(L^2(\maG^x, \lambda_x))$ denotes be the regular representation on
$\maG_x$, given by left convolution, $\pi_x(\phi) \psi := \phi * \psi$
(Definition \ref{def.regular}). We shall use the following notation
throughout the rest of the paper.

\begin{notation}\label{notation.U}
\normalfont{Let us assume that $U \subset M$ is an open,
  $\maG$-invariant subset with $\maG_U \simeq U \times U$ (the pair
  groupoid, see Example \ref{ex.pair}). For any $x_0\in U$, the range
  map then defines a bijection $r : \maG_{x_0} \to U$ and hence a
  measure $\mu$ on $U$ corresponding to $\lambda_{x_0}$. This measure
  does not depend on the choice of $x_0 \in U$ and leads to isometries
  $L^2(\maG_{x_0}, \lambda_{x_0}) \simeq L^2(U; \mu)$ that commute
  with the action of $\maG$.  We then denote by $\pi_0$ the
  corresponding representation of $C^{\ast}(\maG)$ on $L^2(U; \mu)$.
  It is often called the {\em vector representation} of
  $C^{\ast}(\maG)$.  We shall usually write $L^2(U) := L^2(U; \mu)$.
  The representation $\pi_0$ then defines an isomorphism
  $C_r\sp{\ast}(\maG_U) \simeq \pi_0(C_r\sp{\ast}(\maG_U) = \maK$, the
  algebra of compact operators on $L\sp{2}(U)$.}
\end{notation}

\begin{remark}
An important remark is that, since we only consider \emph{Hausdorff}
Lie groupoids, the vector representation $\pi_0$ is injective, by a
result of Khoshkam and Skandalis \cite{KSkandalis}.  In particular, we
have that $1 + \pi_0(a)$ is invertible in $\pi_0(\rCs{\maG})/\maK$ if,
and only if, $1+ a$ is invertible in $\rCs{\maG}/\rCs{\maG_U} $.  We
shall thus identify $C^{\ast}_r(\maG)$ with its image under $\pi_0$,
that is, with a class of operators on $L^2(U)$, without further
comment.
\end{remark}

In the following definition, we shall use that $\rCs{\maG_U}\subset
\ker \pi_x$, $x\notin U$, by definition, and hence that $\pi_x$
descends to the quotient algebra.

\begin{definition}\label{def-fredholm}
We say that $\maG$ is a {\em Fredholm Lie groupoid} if:
\begin{enumerate}[(i)]
\item There is an open, {dense}, $\maG$-invariant subset $U\subset M$
  such that $\maG_U\simeq U\times U$.
\item Given $a \in C^{\ast}_r(\maG)$, we have that $1 + a$ is Fredholm
  if, and only if, all $1 + \pi_x(a)$, $x \in F \smallsetminus U$, are
  invertible (see \ref{notation.U} for notation).
\end{enumerate}
The set $F := M \smallsetminus U$ will be called the set of {\em
  boundary units} of $\maG$ and a set $U$ as in this definition will
be called a manifold with {\em amenable ends}.
\end{definition}

For a Fredholm Lie groupoid $\maG \tto M$, we shall always denote by
$U \subset M$ the $\maG$-invariant open subset from the definition of
a Lie Fredholm groupoid. {Since $U$ is a dense orbit,} it is uniquely
determined by $\maG$. Hence, $F := M \smallsetminus U$ is closed and
$\maG$-invariant.

We shall need the following result in the matricial case.

\begin{proposition}\label{prop.bundles}
 Let us assume that $\maG$ is a Fredholm Lie groupoid, that $M$ is
 compact, and that $e \in M_n(\maC(M))$ is a projection.
\begin{enumerate}[(i)]
 \item The set of representations $\pi_x$, $x\in F$ defines a strictly
   spectral family of representations of $eM_n
   \big(\rCs{\maG}/\rCs{\maG_U} \big) e$.
 \item For any $a \in eM_n(C^{\ast}_r(\maG))e$, we have that $1 +
   \pi_0(a)$ is Fredholm if, and only if, all $1 + \pi_x(a)$, $x \in
   F:=M \smallsetminus U$, are invertible (see \ref{notation.U} for
   notation).
\end{enumerate}
\end{proposition}

\begin{proof}
By Atkinson's theorem (which states that an operator is Fredholm if,
and only if, it is invertible modulo the compacts) and by the
definition of strictly spectral families of representations in the
non-unital case, one can see that condition (ii) in the definition of
Fredhom groupoid is equivalent to the fact that $\maR(\maG_F)$ defines
a strictly spectral family of representations of the quotient
$\rCs{\maG}/\maK \simeq \rCs{\maG}/\rCs{\maG_U} $.  Since $\maG$ is
second countable, its associated $C^*$-algebras are separable, hence
the set $\{\pi_x \vert x \in F\}$ is strictly spectral if, and only
if, it is exhaustive, by Theorem \ref{thm.exhaustive}.  The first part
is then a consequence of Proposition \ref{prop.Morita}.

The second part follows again from the definition of strictly spectral
families of representations in the non-unital case, from the fact that
$\rCs{\maG_U} \simeq \maK(L^2(U))$, and from Atkinson's theorem.
\end{proof}

In the next results, we give an abstract characterization of Lie
Fredholm groupoids, using some of the ideas of the previous
section. We obtain in particular conditions that will make it simpler
to check whether a given Lie groupoid is Fredholm. Recall that, in
this paper, all groupoids are Hausdorff.

\begin{theorem}  \label{thm.Fredholm.Cond}
Assume $\maG$ is a Lie groupoid. If $\maG$ is Fredholm, then following
two conditions are satisfied:
\begin{enumerate}[(i)]
\item The canonical projection $C_r\sp{\ast}(\maG) \to
  C_r\sp{\ast}(\maG_F)$ induces an isomorphism
  \begin{equation*}
     C_r\sp{\ast}(\maG)/C_r\sp{\ast}(\maG_U) \, \simeq \,
     C_r\sp{\ast}(\maG_F) \,, \quad F:=M \smallsetminus U \,.
  \end{equation*}
\item $\maG_F$ has Exel's property.
\end{enumerate}
\end{theorem}

(Recall that condition (ii) means that $\maR(\maG_F)$ is an exhaustive
set of representations of $C_r\sp{\ast}(\maG_F)$.)

\begin{proof}
Let us assume that $\maG$ is a Fredholm groupoid and check the two
conditions of the statement.

Let $p \colon C_r\sp{\ast}(\maG) \to C_r\sp{\ast}(\maG_F)$ be the
natural projection induced by restriction.  It is known that
$C_r\sp{\ast}(\maG_U) \subseteq \ker(p)$ (see the proof of
\cite[Prop.II.4.5 (a)]{renaultBook}).  To prove (i), we need to show
that we have equality $C_r\sp{\ast}(\maG_U) = \ker(p)$.  Let us
proceed by contradiction, that is, let us assume that
$C_r\sp{\ast}(\maG_U) \neq \ker(p)$.  Then we can choose $a = a\sp{*}
\in \ker(p) \smallsetminus C_r\sp{\ast}(\maG_U)$.  Note that since
$\pi_x= \pi_x\circ p$, $x\in F=M\setminus U$, we have $\ker p\subset
\ker \pi_x$, in particular, $1- \pi_x(a) = 1$ is invertible for all $x
\notin U$.  On the other hand, since $a$ is self-adjoint and non-zero,
in the quotient $\ker(p)/C_r\sp{\ast}(\maG_U)$, there is $0 \neq
\lambda \in \RR$ such that $\lambda -a$ is not invertible in
$\ker(p)/C_r\sp{\ast}(\maG_U)$.  By rescaling, we may assume $\lambda
= 1$ and thus $1 -a$ is not invertible in
$C^{\ast}_r(\maG)/C_r\sp{\ast}(\maG_U) \cong C^{\ast}_r(\maG)/\maK$.
We conclude that in this case $\maR(\maG_F)$ is not a strictly
spectral family of representations of $\rCs{\maG}/\rCs{\maG_U}$.  It
then follows that $1-a$ is not Fredholm, however $1- \pi_x(a) = 1$ is
invertible for all $x \notin U$. This is a contradiction.

The first part shows that there is an isomorphism
\begin{equation*}
  C_r^*(\maG_F) \simeq C^{\ast}_r(\maG)/C^{\ast}_r(\maG_U)\simeq
C^{\ast}_r(\maG)/\maK(L^2(U)).
\end{equation*}
By Fredholmness of $\maG$, we have that $\maR(\maG_F)$ defines a
strictly spectral family of representations of $C^{\ast}_r(\maG)/\maK
$, so $\maR(\maG_F)$ is also a strictly spectral family of
representations of $ C_r\sp{\ast}(\maG_F)$. Since $\maG$ is second
countable, it is exhaustive, by Theorem \ref{thm.exhaustive}, so that
$\maG_F$ has Exel's property. This proves (ii) and hence the direct
implication.
\end{proof}

We remark that we proved in fact that, if $\maR(\maG_F)$ defines a
strictly spectral family of representations of
$\rCs{\maG}/\rCs{\maG_U}$, then necessarily
%$\ker(p)= C_r\sp{\ast}(\maG_U)$ so that
$C_r\sp{\ast}(\maG)/C_r\sp{\ast}(\maG_U) \, \simeq \,
     C_r\sp{\ast}(\maG_F)$ and $\maG_F$ has Exel's property.

A stronger form of the converse of Theorem \ref{thm.Fredholm.Cond} is
contained in the following theorem.

\begin{theorem}\label{thm.main.converse}
Let $\maG \tto M$ be a Lie groupoid. Let us assume that $U \subset M$
is a dense $\maG$-invariant open subset such that $\maG_U \simeq U
\times U$. Also, let us assume that the two conditions of Theorem
\ref{thm.Fredholm.Cond} are satisfied. Then, for any unital
$C^{\ast}$-algebra $\mathbfPsi$ containing $M_n(C^{\ast}_r(\maG))$ as
an essential ideal and for any $a \in \mathbfPsi$, we have that
$\pi_0(a)$ if Fredholm if, and only if, the image of $a$ in
$\mathbfPsi/M_n(C^{\ast}_r(\maG))$ is invertible and all $\pi_x(a)$,
$x \in M \smallsetminus U$, are invertible.
\end{theorem}

\begin{proof}
The proof is the same for any $n$ (using also Proposition
\ref{prop.Morita}, so we let $n=1$ for notational simplicity).  Let us
assume that conditions (i) and (ii) of Theorem \ref{thm.Fredholm.Cond}
are satisfied and let $\mathbfPsi$ be a $C^{\ast}$-algebra containing
$C^{\ast}_r(\maG)$ as an essential ideal. Property (i) implies that
$\pi_0$ is injective on $\mathbfPsi$ since $C\sp{\ast}(\maG)$ is an
essential ideal of $\mathbfPsi$.  By our assumptions on $\maG$, we
obtain that the algebra $\Psi/C_r\sp{\ast}(\maG_U) \simeq
\pi_0(\Psi)/\maK =: B$ contains $B_0 :=
\pi_0(C_r\sp{\ast}(\maG))/\maK$ as an ideal and $B/B_0 \simeq
\mathbfPsi/C_r\sp{\ast}(\maG)$. Moreover, $B_0 \simeq
C_r\sp{\ast}(\maG_F)$ by (ii).

Let now $a \in \mathbfPsi$ be arbitrary. By Atkinson's Theorem, we
know that $\pi_0(a)$ is Fredholm if, and only if, its image in $B$ is
invertible. But by (iii) and by Corollary \ref{cor.ideal}, $c \in B$
is invertible if, and only if, the image of $c$ in
$\mathbfPsi/C_r\sp{\ast}(\maG)$ and all $\pi_x(c)$ are invertible for
all $x \in F$.
\end{proof}

\begin{corollary}\label{cor.new}
Let $\maG \tto M$ be a Lie groupoid and $\emptyset \neq U \subset M$
be a dense, $\maG$-invariant, open subset such that $\maG_U \simeq U
\times U$, and $F=M\setminus U$. Then $\maG$ is Fredholm if, and only
if, $\rCs{\maG}/\rCs{\maG_U} \simeq \, C_r\sp{\ast}(\maG_F)$ and
$\maR(\maG_F)$ is a strictly spectral set of representations of
$\maG_F$ (that is, if $\maG_F$ has Exel's property).
\end{corollary}

Recalling the exact sequence in Proposition \ref{renault.exact}, it
follows that, in particular, if $\maG$ is metrically amenable, then
$\maG$ is Fredholdm if, and only if, $\maG_F$ has Exel's property.

\begin{remark}\label{rem.Exel.suff}
It is immediate to observe that, if the set $\maR(\maG_F)$ of regular
representations $\pi_x$, $x \in F$, forms a strictly spectral set of
representations of $\Cs{\maG_F}$, then $\maG_F$ has Exel's strong
property, and conditions (ii) and (iii) of the Theorem
\ref{thm.Fredholm.Cond} are satisfied (using Proposition
\ref{renault.exact}).
\end{remark}

This remark and the above results then give the following sufficient
conditions for a Lie groupoid to be Fredholm.

\begin{corollary}
Let $\maG \tto M$ be a Lie groupoid and $U \subset M$ be a {dense}
$\maG$-invariant open subset such that $\maG_U \simeq U \times U$, and
$F=M\setminus U$.  Assume that $\maR(\maG_F)$ forms a {strictly
  spectral} set of representations of $\Cs{\maG_F}$ or, equivalently,
that $\maG_F$ has Exel's strong property. Then $\maG$ is a Fredholm
groupoid.
\end{corollary}

See also \cite{MantoiuBB}. Together with Theorem \ref{thm.lemma.Exel},
we obtain an important class of Fredholm groupoids. Recall that we
assume that all our groupoids are Hausdorff.

\begin{theorem}\label{thm.main.Fredholm0}
Let $\maG\tto M$ be a \ssub\ Lie groupoid with filtration $U_i$ such
that $U_0$ is dense and $\maG_{U_0} \simeq U_0 \times U_0$. Assume
that all isotropy groups $\maG_x^x$ are amenable. Then $\maG$ is a
Fredholm groupoid.
\end{theorem}

\begin{proof}
By Remark \ref{rem.restriction} applied to $F:= M \smallsetminus U_0$
and by Theorem \ref{thm.lemma.Exel}, we know that the family
$\maR(\maG_F)$ of regular representations of $\maG_F$ is an exhaustive
family of representations of $\Cs{\maG_F}$, that is, $\maG_F$ has
Exel's strong property.  The result follows from the previous
corollary.
\end{proof}

We notice that we can replace $\pi_x$, $x \in F$, with $\pi_b$, $b \in
B_S$, for the strata corresponding to the boundary. We can further
restrict to the set of $b$s in a dense set of orbits, since they give
rise to a faithful family of representations of $\maG_F$, again by
\cite{KSkandalis}.

\subsection{Pseudodifferential operators}
\label{ssec.pseudodifferential}
So far we have discussed mostly groupoids and operators in their
(unitized) $\Cstar$-algebras. The study of differential operators on
suitable non-compact spaces can be reduced to the study of groupoid
algebras using pseudodifferential operators on Lie groupoids. In this
subsection, we will explain this reduction procedure.  See
\cite{hormander3, NicolaRodinoBook, RuzhanskyBook, simanca,
  vishikGrushin} for introductions to pseudodifferential operators.

Let $\maG$ be a Lie groupoid and consider the algebra
$\Psi\sp{*}(\maG)$ of {\em pseudodifferential operators on the
  groupoid $\maG$}, whose definition we now briefly recall \cite{aln2,
  Monthubert, NWX}. Then, for $m \in \ZZ \cup\{\pm \infty\}$,
$\Psi\sp{m}(\maG)$ consists of smooth families $(P_x)_{x \in M}$ of
classical pseudo-differential operators $P_x \in \Psi^m(\maG_x)$ of
order $m$, that are right invariant with respect to the action of
$\maG$ and have compactly supported distribution kernels.  In
particular, $\Psi\sp{-\infty}(\maG)$ is nothing but the convolution
algebra of smooth, compactly supported function on $\maG$, that is,
$\Psi\sp{-\infty}(\maG) \simeq \CIc(\maG)$.

This construction extends right away to operators between smooth
sections of some vector bundles $E, F \to M$. There are two methods of
doing this. The first method is to consider families $(P_x)_{x \in M}$
of operators $P_x \in \Psi^m(\maG_x; r^*(E), r^*(F))$. We denote by
$\Psi^m(\maG; E, F)$ the resulting set of operators. By embedding $E
\oplus F$ into a trivial bundle of dimension $N$, we can identify
$\Psi^m(\maG; E, F)$ with a subspace of $M_N(\Psi^m(\maG))$, when
convenient. Similar considerations yield differential operators
generated by $\maV$ acting between vector bundles $E$ and $F$ endowed
with metric preserving connections.

The second method to deal with operators between two vector bundles
$E, F \to M$ is as follows. Let us choose an embedding of $E \oplus
F\to \underline{\CC}^N$ into a trivial vector bundle and denote by $e$
and $f$ the corresponding orthogonal projections onto the images of
$e$ and $f$.  Then we can identify
\begin{equation*}
 \Psi^m(\maG; E, F) \ede f M_N(\Psi^m(\maG) ) e \,.
\end{equation*}
This idea applies to other types of similar operators (norm closures of 
operators in $\Psi^m(\maG)$, differential operators, ... ).

One of the main points of considering pseudodifferential operators on
groupoid is that, if we denote $\maV := \Lie(\maG) :=
\varrho(\Gamma(A(\maG)))$, then
\begin{equation}
 \Diff^m(\maV; E, F) \seq \Psi^m(\maG; E, F) \cap \Diff(M; E, F)\,,
\end{equation}
where $\Diff^m(\maV; E, F)$ is generated by $\CI(M; 
\End(E \oplus F))$
and $\maV$ (using compatible connections on $E$ and $F$) 
and $\Diff(M; E, F)$ denotes {\em all} differential operators on 
$M$ acting between sections of $E$ and $F$. See also \ref{rem.aln1}.

Another main point for introducing the algebra $\Psi\sp{*}(\maG)$ is
related to parametrices and inverses of differential operators.  To
discuss this, let us denote by $\overline{\Psi}(\maG)$ the
$C\sp{\ast}$-algebra obtained as the closure of $\Psi\sp{*}(\maG)$
with respect to all contractive $*$-representations of
$\Psi\sp{-\infty}(\maG)$, as in \cite{LNGeometric}.  As we will see
below, this definition extends to operators acting between vector
bundles. Let us denote, as usual, by $S\sp{*}A$ the set of unit
vectors in the Lie algebroid $A\sp{*}(\maG)$ associated with $\maG$
(as in Remark \ref{rmk.liealgebroid}), with respect to some fixed
metric on $A\sp{*}(\maG)$. Then $\overline{\Psi}(\maG)$ fits into the
following exact sequence
\begin{equation}\label{main-seq}
 0 \to C\sp{\ast}(\maG) \to \overline{\Psi}(\maG)
 \overset{\sigma_0}{-\!\!\!\to} \maC_0(S\sp{*} A) \to 0 \,.
\end{equation}
(See for instance \cite{LMN1} and the references therein.) Moreover,
if $P \in \Psi\sp{m}(\maG; E, F)$, $m \ge 0$, is elliptic and
invertible in $L^2$, then its inverse will be in
$\overline{\Psi}(\maG; F, E)$. Typically, Fredholm conditions are
obtained for Fredholm Lie groupoids by applying our results to the
algebra $\mathbfPsi = M_N(\overline{\Psi}(\maG))$, for some $N$. Let
us see how this is done.

We first need to recall the definition of Sobolev spaces.  Let us fix
in this subsection a Lie groupoid $\maG \tto M$. For the purpose of
the next result, let us assume that its space of units $M$ has an
open, dense, $\maG$-invariant subset $U \subset M_0 := M
\smallsetminus \pa M$ such that the restriction $\maG_U$ is isomorphic
to the pair groupoid $U \times U$. Let us also assume that the space
of units $M$ is {\bf compact.}  This is needed in order to construct
canonical Sobolev spaces on the interior of $M$. Indeed, there is an
essentially unique class of metrics on $A(\maG)$, which, by
restriction, gives rise to a class of metrics on $U$, that are called
{\em compatible} (with the groupoid $\maG$), see also \cite{aln1}. All
these metrics are Lipschitz equivalent and complete \cite{sobolev,
  LNGeometric}.  In fact, the Sobolev spaces of all these metrics will
coincide. They are given as the domains of the powers of $1 + \Delta$,
where $\Delta$ is the (geometer' s, i.e. positive) Laplacian. We shall
denote by $H\sp{s}(U) = H\sp{s}(M)$ these Sobolev spaces. The Sobolev
spaces $H\sp{s}(M)$ are discussed in detail in \cite{sobolev}.  See
also \cite{frascatti} for a review. The operators in $\Psi^m(\maG)$
(and their vector bundle analogues) model differential operators
associated to (any) compatible metric (see also Remark \ref{rem.aln1}
below).  By considering a vector bundle $E \to M$ and $\Delta_E$ the
associated Laplacian on sections of $E$, we obtain the spaces $H^s(M;
E)$.

Note that the Sobolev spaces $H^s(M)$ are {\em not} defined using the
compact manifold structure on $M$, but rather using the complete
metric on $U \subset M$. Recall that if the set $M\setminus U$ is as
in Definition \ref{def-fredholm}, then the set $U$ is called a
manifold with {\em amenable ends.} If $P \in \Psi^m(\maG; E, F)$, then
$P : H\sp{s}(M; E) \to H\sp{s-m}(M; F)$ is continuous.

Let $\Psi^m(\maG; E, F)$ denote the pseudodifferential operators
acting on the lifts of $E$ and $F$ to $\maG$ via $d$.  We let
$\Psi^m(\maG; E) := \Psi^m(\maG; E, E)$.  Then if $P : H^s(M; E) \to
H^{s-m}(M; F)$ is an order $m$ pseudodifferential operator, we replace
the study of its Fredholm properties with those of
\begin{equation}\label{eq.def.P1}
 P_1 \ede \left [
 \begin{array}{ccc}
  0 & P & 0 \\
  P^* & 0 & 0 \\
  0 & 0 & Q_m
 \end{array}
 \right ] \, ,
\end{equation}
where $Q_m$ is an invertible order $m$ pseudodifferential operator
acting on the complement of $E \oplus F$.  Let $G$ denote the
complement of $E \oplus F$ in $\underline{\CC}^N$.  We fix the smooth
vector bundles $E, F, G \to M$ in what follows. We denote by
$\Delta_E$, $\Delta_F$, and $\Delta_G$ the associated Laplacians
acting on sections of $E$, respectively, $F$ and $G$.  For instance,
$Q_m \in \Psi^m(\maG; G)$ could be $(1 + \Delta_G)^{m/2}$ (or a
suitable approximation it, see \cite{LNGeometric}).

One of the main drawbacks of the algebras $\Psi^m(\maG)$ is that they
are too small to contain resolvents. This issue is easily fixed by
considering completions with respect to suitable norms. Let us
consider the norm $\| \, \cdot \, \|_{m,s}$ defined by.
\begin{equation}\label{eq.def.ms}
  \|P\|_{m, s} \ede \|(1 + \Delta_F)\sp{(s-m)/2} P 
  (1 + \Delta_E)\sp{-s/2} \|_{L^2 \to L^2}
\end{equation}
We then let
\begin{equation}\label{eq.def.Cms}
 L^{m}_{s}(\maG; E, F) \ede \| \, \cdot \, \|_{m,s}\mbox{-closure of } 
 \Psi^m(\maG; E, F)\,.
\end{equation}
Recall that $H^{s}(M; E)$ is the domain of $(1 + \Delta_E)^{s/2}$, if
$s \ge 0$, whenever $M$ is compact (see \cite{sobolev,
  Grosse.Schneider.2013, LNGeometric}. Hence $L^{m}_{s}(\maG; E, F)$
is the {\em norm closure} of $\Psi^m(\maG;E,F)$ in the topology of
continuous operators $H^s(M; E) \to H^{s-m}(M; F)$.  We note that, if
$P \in \Psi^m(\maG; E, F)$, then
\begin{equation}
   (1 + \Delta_F)\sp{(s-m)/2} P (1 + \Delta_E)\sp{-s/2} \; \in \;
  \overline{\Psi}(\maG; E, F) \, =:\, L^{0}_{0}(\maG; E, F)\,
\end{equation}
by \cite{LNGeometric} (see also \cite{alnv}). Moreover, let 
\begin{equation}
 \maW^{m}(\maG) \ede \Psi^{m}(\maG) + \cap_{s} L^{-\infty}_{s}(\maG)\,.
\end{equation}
Then $\maW^{m}(\maG) \subset L^{m}_{s}(\maG)$ and
$\maW^{\infty}(\maG)$ is an algebra of pseudodifferential operators
that contains the inverses of its $L^2$-invertible operators.

Recall that we denote by $\pi_0 : C\sp{\ast}(\maG) \to
\maL(L\sp{2}(M))$ the vector representation and that $M$ is
compact. It is unitarily equivalent to the regular representations
$\pi_x$, $x \in U$.  Moreover, since $\maG$ is Hausdorff, $\pi_0$ is
injective.  We have the following result from \cite{LNGeometric}.

\begin{proposition}\label{prop.nonclassical2}
Let $\maG \tto M$ be a Lie groupoid with $\maG_U = U \times U$ and $M$
compact. Then $M_N(\overline{\Psi}(\maG)) = M_N(L_0^0(\maG))$ contains
$M_N(C\sp{\ast}(\maG))$ as an essential ideal. Let $s \in \RR$ and $P
\in L^m_s(\maG; E, F) \supset \Psi\sp{m}(\maG; E, F)$. Then
\begin{equation*}
 a \ede (1 + \Delta_F)\sp{(s-m)/2} P (1 + \Delta_E)\sp{-s/2} \in
\overline{\Psi}(\maG; E, F),.
\end{equation*}
We have that $P : H\sp{s}(M; E) \to H\sp{s-m}(M; F)$ is Fredholm if,
and only if, $a : L\sp{2}(M; E) \to L^2(M; F)$ is Fredholm.
\end{proposition}
This proposition applies, in particular, if $\maG$ is a Fredholm Lie
groupoid with $M$ compact.

\begin{proof}
We can assume $N = 1$. The fact that $\overline{\Psi}(\maG)$ contains
$C\sp{\ast}(\maG)$ as an essential ideal is a general fact--true for
any Lie groupoid. This general fact is true because it is true for any
non-compact manifold, in particular, for each of the manifolds
$\maG_x$. We have, by the definitions of Sobolev spaces and of
Fredholm operators, that $P$ is Fredholm if, and only if, $a$ is
Fredholm. The fact that $a := (1 + \Delta_F)\sp{(s-m)/2} P (1 +
\Delta_E)\sp{-s/2} \in \overline{\Psi}(\maG; E, F)$ follows from the
definition of $L_s^m(\maG; E, F)$, Equation \eqref{eq.def.Cms} and the
results in \cite{LNGeometric}, as noticed already above.
\end{proof}

This proposition then gives right away the following result (for
technical reasons, we may have to replace $P$ with $a$ in certain
proofs). Recall that an operator $P \in \Psi\sp{m}(\maG; E, F)$
consists of a right invariant family $P = (P_x)$, $x \in M$, with
$P_x$ acting between the sections of $r^*(E)$ and $r^*(F)$ on
$\maG_x$. For simplicity, we shall drop $r^*$ from the notation below,
since there is no danger of confusion.  We have $P_x = \pi_x(P)$. The
operators $P_x$, $x\in M\setminus U$, will be called \emph{limit
  operators} (of $P$). They go back to \cite{Favard}.

\begin{theorem}\label{thm.nonclassical2}
Let $\maG \tto M$ be a Fredholm Lie groupoid with $M$ compact and
$\emptyset \neq U \subset M$, open such that $\maG_U = U \times
U$. Let $s\in \RR$ and $P \in L^m_s(\maG; E, F) \supset
\Psi\sp{m}(\maG; E, F)$.  We have
\begin{multline*} %\label{eq.Fredholm.c2}
	P : H^s(M; E) \to H^{s-m}(M; F) \mbox{ is Fredholm}
        \ \ \Leftrightarrow \ \ P \mbox{ is elliptic and each }\\
	\ P_{x} : H^s(\maG_x; E) \to H^{s-m}(\maG_x; F)\,, \ x \in M
        \smallsetminus U\,, \ \mbox{ is invertible} \,.
\end{multline*}
\end{theorem}

\begin{proof}
Let us use the notation of Proposition \ref{prop.nonclassical2}. By
replacing $P$ with $a$ and using Proposition \ref{prop.bundles}, we
may assume that we work with $N \times N$ matrices of operators.  The
proof is the same for all $N$, so we let $N=1$. Since $\maG$ is
Fredholm, Theorem \ref{thm.main.converse}, applied to $\mathbfPsi :=
\overline{\Psi}(\maG)$, gives that $a \in \overline{\Psi}(\maG)$ is
Fredholm if, and only if, its image in
$\overline{\Psi}(\maG)/C\sp{\ast}(\maG)$ is invertible and all the
operators $\pi_x(a)$ are invertible. We then notice that $\pi_x(a) =
(1 + \Delta_x)\sp{(s-m)/2} P_x (1 + \Delta_x)\sp{-s/2}$ since the
extension of $\pi_x$ to operators affiliated to
$\overline{\Psi}(\maG)$ is given by $\pi_x(P) = P_x$ since this is
true for $P \in \Psi\sp{0}(\maG)$ and $\pi_x(\Delta) = \Delta_x$, the
Laplacian on $\maG_x$ by \cite{LNGeometric}.
\end{proof}

\begin{remark}\label{rem.explicit1}
To obtain the result as formulated in Theorem \ref{thm.nonclassical},
we let $\alpha$ be the set of orbits of $\maG$ acting on $F := M
\smallsetminus U$, $M_0 := U$, $M_\alpha := \maG_x$, for some $x$ in
the orbit $\alpha$, $G_\alpha := \maG_x^x$, and $P_\alpha :=
\pi_x(P)$. It may be useful to notice here that often the bundle
$M_\alpha \to M_\alpha /G_\alpha$ is trivial. This is the case, for
example, for \ssub\ groupoids.
\end{remark}

\begin{remark}\label{rem.aln1}
Let $\maV := r_*(\Gamma(A(\maG))$ be the vector fields on $M$ coming
from the infinitesimal action of $\maG \tto M$ on $M$. We denote by
$\Diff^m(\maV)$ the set of {\em order $m$} differential operators on
$M$ generated by $\maV$ and multiplication with functions in $\CI(M)$
and $\Diff(\maV) = \cup_m \Diff^m(\maV)$. We denote by $\Diff^m(\maV;
E, F)$ the analogous differential operators acting between sections of
vector bundles $E, F \to M$.  We have that all geometric operators
(Laplace, Dirac, Hodge, ... ) associated to a compatible metric on $M$
(one that comes by restriction from $A(\maG)$) belong to
$\Diff^m(\maV; E, F)$ for suitable $E, F \to M$; moreover, $\pi_x(D)$
for a geometric operator is of the same type as $D$. \cite{aln1}.
\end{remark}

\begin{remark}\label{rem.explicit2}
If $\maG$ is $d$-connected (in the sense that the fibers of $d: \maG
\to M$ are connected), then the orbits of the vector fields $\maV$ are
the same as the orbits of $\maG$. If $x \in M$ and $\alpha \subset M$
is its orbit, then this orbit (with its intrinsic manifold topology)
is diffeomorphic to the set $M_\alpha/G_\alpha$. This makes more
explicit the data in Remark \ref{rem.explicit1}.
\end{remark}

Although we shall not use this in the present paper, let us record the
consequence for essential spectra. Notice that we do not need the
closure of the unions for the essential spectra. We denote by
$\sigma(Q)$ the spectrum of an operator $Q$, defined as the set of
$\lambda \in \CC$ such that $Q - \lambda$ is {\em not} invertible, and
by $\sigma_{\ess}(Q)$ its essential spectrum, defined as the set of
$\lambda \in \CC$ such that $Q - \lambda$ is {\em not} Fredholm. Here
$Q$ may be unbounded. If $P \in \Psi\sp{m}(\maG; E)$, $m > 0$, we
consider it to be defined on $H^m$.

\begin{corollary} Consider the framework of Theorem 
 \ref{thm.nonclassical2} and $P \in \Psi\sp{m}(\maG; E)$, $m \ge 0$,
 elliptic. We have
\begin{equation*}
\begin{gathered}
 \sigma_{\ess}(P) \seq \cup_{x \in \pa M} \sigma(P_x) \cup \cup_{\xi
   \in S^*A} \sigma(\sigma_0(P)(\xi))\,, \quad \mbox{if } m =
 0\,,\\ \sigma_{\ess}(P) \seq \cup_{x \in \pa M} \sigma(P_x)\,, \quad
 \mbox{if } m > 0\,.
\end{gathered}
\end{equation*}
\end{corollary}

Thus, in the scalar, order zero case, we have
\begin{equation}
 \sigma_{\ess}(P) \seq \cup_{x \in \pa M} \sigma(P_x) \cup
 Im(\sigma_0(P))\,,
\end{equation}

Combining Theorems \ref{thm.nonclassical2} and
\ref{thm.main.Fredholm0}, we obtain the following result (recall that
all our groupoids are Hausdorff and second countable).

\begin{theorem}\label{thm.main.Fredholm}
Let $M$ be compact and $\maG \tto M$ be a \ssub\ Lie groupoid with
filtration $U_i$ such that $U_0$ is dense in $M$ and $\maG_{U_0}
\simeq U_0 \times U_0$.  Assume that all isotropy groups $\maG_x^x$
are amenable. Then, for any smooth vector bundles $E, F \to M$ and any
$P \in L^m_s(\maG; E, F) \supset \Psi\sp{m}(\maG; E, F)$, we have
\begin{multline*} %\label{eq.Fredholm.c2}
	P : H^s(M; E) \to H^{s-m}(M; F) \mbox{ is Fredholm}
        \ \ \Leftrightarrow \ \ P \mbox{ is elliptic and all}\\
	\ P_{x} : H^s(\maG_x; E) \to H^{s-m}(\maG_x; F)\,, \ x \in M
        \smallsetminus U\,, \ \mbox{ are invertible}\,.
\end{multline*}
\end{theorem}

We actually have $P_{x} : H^s(\maG_x; r^*(E)) \to H^{s-m}(\maG_x;
r^*(F))$, but, as we have explained above, we drop $r^*$ from the
notation, for the sake of simplicity, since there is no danger of
confusion. We will continue to do that in all related statements
below.

\begin{remark}\label{rem.ghost}
If $\maG$ be as in the above theorem (a \ssub\ groupoid with
filtration $U_i$), and assume that it is also $d$-connected, as in
Remark \ref{rem.explicit1}. Then, when $P=D$ is a differential
operator, we can identify the limit operators $D_\alpha$ (and hence
also $M_\alpha$ and $G_\alpha$) in a direct way. Indeed, let $\alpha$
be the $\maV$-orbit = the $\maG$-orbit through some $x \in M
\smallsetminus U_0$. We have $G_\alpha = \maG_x^x$, as discussed in
Remark \ref{rem.explicit1}. Then $M_\alpha = \alpha \times G_\alpha$.
Let $D \in \Diff(\maV)$. By linearity, we may assume $D = a X_{i_1}
X_{i_2} \ldots X_{i_m}$, with $X_i \in \maV$ a {\em local basis} of
$\maV$ near $x$ and $a \in \CI(M)$. (We may also assume that $i_j$
form a non decreasing sequence, but that is not necessary.) Since
$\alpha$ is an orbit, we have that this basis $(X_j)$ consists of
vectors tangent to $\alpha$. We can choose this basis such that $X_1,
\ldots , X_k$ form a local basis of $T_x \alpha$ and $X_j$, $j > k$,
are zero on the orbit $\alpha$ (where they are defined). Consequently,
the vector fields $X_j$, $j > k$, come from a basis $\tilde X_j$, $j >
k$, of $Lie (G_\alpha)$. We then have
 \begin{equation}\label{eq.def.indicial}
  D_\alpha \seq a\vert_{\alpha} Y_{1} Y_{2} \ldots Y_{m}\,, \ \ \mbox{
    near } \{x\} \times G_\alpha \subset \alpha \times G_\alpha =:
  M_\alpha\,,
 \end{equation}
where $Y_j = X_{i_j}\vert_{\alpha}$, if $i_j \le k$, and $Y_j =
\widetilde {X}_{i_j}$ if $i_j > k$. We shall say that the variables
$Y_j$ (or $X_j$) with $j > k$ are {\em ghost derivatives} at the
considered orbit $\alpha$ (see Remark \ref{rem.ghost}).
\end{remark}

See also \cite{KarstenCR, KarstenBMB, BKSo2, vassout}.
It would be interesting to extend the results of this section to
$L^p$-spaces in view of \cite{ammann.grosse:16b, RochBookLimit}.

\section{Examples}\label{sec.ex}

In this section, we use the results of the previous section to obtain
Fredholm conditions for operators on some standard non-compact
manifolds: manifolds with (poly)cylindrical ends, asymptotically
Euclidean manifolds, asymptotically hyperbolic manifolds, boundary
fibration structures, and others. We tried to write this section in
such a way that it can to a large extent be read independently of the
rest of the paper.  The reader interested only in applications can
start reading the paper with this section.

The general setting of this section is that of a (non-compact) smooth
manifold $M_0$ whose geometry is determined by a compactification $M$
to a manifold with corners and a Lie algebra of vector fields $\maV$
on $M$. (Although we shall not use this, let us mention for people
familiar with the concept, that $(M, \maV)$ will be a Lie manifold
with some additional properties.) The differential (and
pseudodifferential) operators considered and for which we obtain
Fredholm conditions are the ones generated by $\maV$ and $\CI(M)$,
that is, the ones in $\Diff(\maV)$ and its variants for vector
bundles. (Recall that $\Diff(\maV)$ was introduced in Remark
\ref{rem.aln1}.)  The proof of the Fredholm conditions in this section
are obtained as particular cases of Theorem \ref{thm.nonclassical2} by
showing that appropriate groupoids are Hausdorff \ssub\ Lie groupoids
with amenable isotropy groups and hence that they are Fredholm
groupoids, in view of Theorem \ref{thm.lemma.Exel}. Many of the
results below can also be obtained from the results in \cite{LMN1,
  LNGeometric}, but
  % \sout{this turned out to be not so} \blue
  {the approach followed here aims to be more} convenient for
non-specialists.

In this section, we continue to denote by $\maG$ a Lie groupoid with
units $M$, a manifold of dimension $n$.  In the theorem yielding
Fredholm conditions, $M$ will be assumed {\bf compact} and endowed with a
{\em smooth} metric $h$.  We assume that $h$ is defined {\em
  everywhere} on $M$, in particular, that it extends to a smooth
manifold containing $M$ as a submanifold. Also, usually it will be no
loss of generality to assume $M$ connected. Our results are formulated 
for operators in $L_s^m(\maG; E, F)$, which is a suitable completion of
$\Psi^m(\maG; E, F)$, see Equations \eqref{eq.def.ms} and 
\eqref{eq.def.Cms}. This allows us to greatly enlarge the scope of 
our results since the completion procedure defining the spaces
$L_s^m$ leads to algebras that are closed under taking the inverses
of $L^2$-invertible elements.

\subsection{Examples related to group actions}
We include first some examples that are closely related to group
actions.

\subsubsection{The action of a group on a space}
Let us assume that Lie group $G$ acts smoothly on a manifold $M$. This
yields the \emph{transformation} (or {\em group action}) Lie groupoid
$\maG := M \rtimes G$, which, as a set, consists of $M \times G$ and
has units $Ma$. We have
\begin{equation*}
 d(x, g) \ede x\,, \ r(x, g) \ede g x\,, \ \mbox{
   and }\ (h x, g) (x, h) := (x, gh)\,.
\end{equation*}
If $M = G$ with $G$ acting by translations, then $M \rtimes G \simeq G
\times G$, the product groupoid. Therefore, if $G \subset M$ as a
dense, $G$-invariant open set, and $G$ acts on itself by left
translation, we are in the setting in which we can ask if the
resulting groupoid is Fredholm Lie groupoid. This groupoid was used in
\cite{GeorgescuNistor2,mantoiuGGroups,mougel1}. The groupoids used in
\cite{GeorgescuNistor2,mougel1} turn out to be \ssub\ groupoids. As
discussed in those papers, this recovers the classical HVZ-theorem
\cite{CFKS, Simon4, TeschlBook} as a particular case of Theorem
\ref{thm.nonclassical2} (with $U = G$).  We note that $M \times G$ is
always Hausdorff (since $M$ and $G$ are Hausdorff). If $V \subset M$
is an open subset, then the reduction groupoid $(M \rtimes G)_V^V$
will be called a {\em local transformation (or action)} groupoid, and
will also be Hausdorff. Many related results (including the non Lie
case), were also obtained by Bottcher, Chandler-Wilde, Karlovich,
Lindner, Rabinovich, Roch, Rozenblyum, Silberman, and many others. See
\cite{BastosFernandesKarlovich,BottcherKarlovichSpitkovsky,
  BottcherToeplitzBook, LindnerMem, Rozenblyum96, rabinovichRochGe08,
  RochBookLimit, Rab3, RochBookNGT} and the references therein.

This groupoid models pseudodifferential operators compatible with any
$G$ invariant metric on $G$. Let $\mathfrak{g}$ be the Lie algebra of
$G$. Then $A(\maG) = M \times \mathfrak{g}$, with $\mathfrak{g}$
acting on $M$ via the infinitesimal action of $G$. The associated Lie
algebra of vector fields is $\Lie(\maG) := \varrho(\Gamma(A(\maG))) =
\maV := \CI(M) \mathfrak{g} \subset \Gamma(T\Omega)$.  We have that
$\Diff(\maV)$ is generated by $\CI(M)$ and $\mathfrak{g}$.

\subsubsection{The $b$-groupoid}
Let $M$ be a manifold with corners. Then the $b$-groupoid of
\cite{LNGeometric,MelroseAPS, Monthubert,NWX} is defined as a set as
the disjoint union
\begin{equation}
 \maG_b \ede \sqcup_{F} (F \times F) \times (\RR_+^*)^{k_F}
 \simeq \sqcup_{F} (F \times F) \times \RR^{k_F} \tto \sqcup_{F} F \seq M \ ,
\end{equation}
where $F$ ranges through the open, {\em connected} faces of $M$, $k_F$
is the codimension of the face $F$, $F\times F$ is the pair groupoid
and $(\RR_+^*)^{k_F}$ is a group for componentwise multiplication (and
hence also a Lie groupoid).
  
To obtain the smooth structure on this groupoid, we notice that,
locally, it is a transformation groupoid $[0, \infty)^k \rtimes
  (\RR_+^*)^k$. More precisely, let us choose $g \in \maG_b$. Then $g
  = (x, y, v) \in (F \times F) \times \RR^{k_F}$ for some {\em
    connected} open face $F \subset M$. We can choose a coordinate
  system $V \subset M$ such that $x, y \in V$, $\overline{V}$ is
  compact in $F$, and we have a tubular neighborhood $V \times [0,
    \epsilon)^{k_F} \subset M$.  Then, a neighborhood of $g$ in
    $\maG_b$ is diffeomorphic to the local transformation groupoid
    obtained by reducing to $V \times [0, \epsilon)^{k_F}$ the action
      groupoid
\begin{equation*}
   (\RR^{d_F} \times [0, \infty)^{k_F}) \rtimes (\RR^{d_F} \times (0,
    \infty)^{k_F})\,,
\end{equation*}
where $d_F = n-k_F$ is the dimension of $F$. Here $\RR^{d_F}$ acts by
translations on itself and $\RR_+^* = (0, \infty)$ acts by
multiplication on $[0, \infty)$. In applications, it will be, in fact,
  more convenient to notice that $\RR^{d_F} \rtimes \RR^{d_F}$ is the
  pair groupoid, and hence to identify a neighborhood of $g$ with a
  reduction of
\begin{equation*}
   (V \times V) \times [0, \infty)^{k_F} \rtimes (0, \infty)^{k_F}\,,
\end{equation*}
the product of the pair groupoid $V \times V$ and the {transformation}
groupoid $[0, \infty)^{k_F} \rtimes (0, \infty)^{k_F}$.

If $M$ has embedded faces, that is, if each hyperface $H$ has a
defining function $r_H$, then we can identify $\maG_b$ with an open
subset of Monthubert's realization of the $b$-groupoid
\cite[Proposition 4.5]{Monthubert03}
\begin{equation*}
  \maM \ede \{\, (x,y,t)\in M\times M\times [-1,1]^\maH |\,
  (1-t_K)x_K(x) =(1+t_H)x_H(y)\, \}\,,
\end{equation*}
where $\maH$ denotes the set of {\em hyperfaces} of $M$ and $H, K \in
\maH$.  See also \cite{Lesch, VertmanLesch15, RouseLescureMont,
  RouseLescure, vanErpYuncken}.

We see that the $b$-groupoid is a \ssub\ gropoid as follows.  The set
$U_k$, $k \le n$, is defined as the union of the open faces of
codimension $k$ of $M$. Then $(\maG_b)_{U_k\setminus U_{k-1}} $ is
isomorphic to the topological disjoint union of the groupoids $(F
\times F) \times \RR^{k}$, where $F$ ranges through the set of open
faces of codimension $k$. In particular, $(\maG_b)_{U_k\setminus
  U_{k-1}} $ is isomorphic to the fibered pull-back of a bundle of Lie
groups, by Example \ref{ex.help-for-b}.  In particular,
$(\maG_b)_x^x\cong \RR^{n-k}$ is amenable, for all $x\in M$.

\subsubsection{Manifolds with poly-cylindrical ends}
The Lie algebroid of $\maG_b$ is identified by
\begin{equation}\label{eq.def.Vb}
 \Gamma(A(\maG_b)) \simeq \maV_b \ede \{ X\in \Gamma(M; TM)|\, X
 \mbox{ tangent to all faces of } M\}\,.
\end{equation}
Let $h$ be an ordinary metric on $M$. The general form of a compatible
metric on $M$ is then
\begin{equation}\label{eq.def.bmetric}
 g_b \ede h + \sum_{H \in \maH} \left( \frac{dx_H}{x_H} \right)^2 \,.
\end{equation}
Manifolds with metrics of this form will be called {\em manifolds with
  poly-cylindrical ends}, following \cite{MelroseScattering}.  The Lie
algebroid $A(\maG_b)$ is often denoted $T^bM$. The groupoid $\maG_b$
models pseudodifferential operators compatible with the metric
$g_b$. Thus, by definition, the metric $g_b$ comes from restriction
from the Lie algebroid $A(\maG_b)$, since $A(\maG_b) \vert_{M_0} =
TM_0$. Since the base $M$ is compact, all metrics on $A(\maG_b)$ will
be equivalent, so the constructions will not depend (essentially) on
the choice of the metric.  In particular, all geometric operators
associated to the metric $g_b$ (Laplace, Dirac, Hodge, ... ) will
belong to $\Psi^m(\maG_b; E, F)$ (which is independent of the metric),
for suitable $E$ and $F$. Moreover, it turns out that Theorem
\ref{thm.main.Fredholm} applies to $\maG_b$ and $\Psi^m(\maG_b; E,
F)$, provided that {\em $M$ is compact.}  (See also Remark
\ref{rem.aln1}. Further details can be found in \cite{aln1}.)

The statement of Theorem \ref{thm.Fredholm.Cond} can be (slightly)
simplified in this case by noticing the following. The representations
$\pi_x$ and $\pi_y$ are unitarily equivalent if $x$ and $y$ are in the
same open face $F$, in which case, they will act on $F \times
\RR^{k_F} \simeq \maG_x$. Let $\pi_F$ be the associated
representation. For Fredholm conditions, it is enough to consider the
invertibility of the operators $P_H := \pi_x(P)$, $x \in H$, for the
faces of maximal dimension (that is, for hyperfaces) in order to
obtain Fredholm conditions, since $\pi_x$ is weakly contained in
$\pi_y$ if $x$ is contained in the closure of the face containing
$y$. Let $\maH$ denote the set of hyperfaces of $M$, as before.
Recall that $L^m_s(\maG; E, F)$ is a suitable completion of
$\Psi\sp{m}(\maG; E, F)$, see Equations \eqref{eq.def.ms} and
\eqref{eq.def.Cms}. We obtain the following result
\cite{MelrosePiazza}.

\begin{corollary}\label{cor.pce}
Let $P \in L_s^m(\maG_b; E, F) \supset \Psi^m(\maG_b; E, F)$, $M$
compact. We have
\begin{multline*} %\label{eq.Fredholm.c2}
	P : H^s(M; E) \to H^{s-m}(M; F) \mbox{ is Fredholm}
        \ \ \Leftrightarrow \ \ P \mbox{ is elliptic and all}\\
	\ P_{H } : H^s(H \times \RR; E) \to H^{s-m}(H \times \RR;
        F)\,, \ H \in \maH \,, \ \mbox{ are invertible}\,.
\end{multline*}
All geometric differential operators associated to the metric $g_b$
belong to $\Diff^m(\maV_b; E, F) \subset \Psi^m(\maG_b; E, F)$, for
suitable vector bundles $E, F \to M$.
\end{corollary}

If $P \in \Diff^m(\maV_{b}; E, F)$, the limit operators $P_x$, $x \in
\pa M$, are obtained as explained in Remark \ref{rem.ghost}.  They are
invariant with respect to the action of the isotropy group $\maG_x^x $
(this is always the case for operators of the form $\pi_x(a)$).  In
this example, only derivatives of the form $x_H \pa_{x_H}$ are ghost
derivatives (see Remark \ref{rem.ghost}).

\begin{remark}
 Carvalho and Qiao have constructed in \cite{carvalhoYu} a similar
 groupoid to the $b$-groupoid in order to study layer
 potentials. Their groupoid, however, was not $d$-connected, in
 general. Nevertheless, the above corollary generalizes to their
 setting, after some obvious modifications.
\end{remark}

\subsubsection{Asymptotically Euclidean spaces}
Let us assume that $M$ has a smooth boundary $\pa M$ with defining
function $x = x_{\pa M}$ and let $\maV_{sc} := r \maV_b$.  The
resulting differential operators $\Diff(\maV_{sc})$ and the associated
pseudodifferential ooperators are the SG-operators of
\cite{Coriasco,Parenti, SchroheSymbol, SchroheFrechet, SchroheSG}
(called ``scattering operators'' in \cite{MelroseScattering}). They
can be obtained by considering the groupoid
\begin{equation}
  \maG_{sc} \ede TM\vert_{\pa M} \sqcup (M_0 \times M_0) \tto \pa M
  \sqcup M_0 \seq M \,.
\end{equation}
To obtain a manifold structure on $\maG_{sc}$, let us consider first
$G = \RR^n$ and $M=$ the radial compactification of $G$ with the
induced action of $G$. Then $\maG_{sc} = M \rtimes G$ is Hausdorff and
a \ssub\ groupoid. In general, $\maG_{sc}$ is locally of this form
(and can be obtained by glueing reductions of such groupoids), and
hence it is Hausdorff.  It satisfies $Lie(\maG_{sc}) :=
\Gamma(A(\maG_{sc})) = \maV_{sc}.$ Thus, if $g_b$ denotes a $b$-metric
on $M$ (or rather, on $M_0$, then the natural metric associated to
$\maV_{sc}$ is
\begin{equation}\label{eq.def.scmetric}
  g_{sc} \ede x^{-2} g_b \ede x^{-2} \big (h + x^{-2} dx^2 \big ) \,,
\end{equation}
where $g_b$ is as in Equation
\eqref{eq.def.bmetric}.

We have that the orbits of $\maG$ on $F := M \smallsetminus U_0 = \pa
M$ are reduced to points and that each stabilizer $\maG_x^x = T_x M
\simeq \RR^n$, for $x \in F$. In particular, $\maG_x = \maG_x^x$, for
$x \in \pa M$ and all derivatives at the boundary are ghost
derivatives.  If $P \in \Diff^m(\maV_{sc}; E, F)$, the limit operators
$P_x$ are obtained as explained in Remark \ref{rem.ghost}.  Theorem
\ref{thm.main.Fredholm} becomes in our case:

\begin{corollary}\label{cor.sc}
Let $P \in L_s^m(\maG_{sc}; E, F) \supset \Psi^m(\maG_{sc}; E, F)$,
$M$ compact. We have
\begin{multline*} 
	P : H^s(M; E) \to H^{s-m}(M; F) \mbox{ is Fredholm}
        \ \ \Leftrightarrow \ \ P \mbox{ is elliptic and all}\\
	\ P_{x} : H^s(T_x M; E) \to H^{s-m}(T_x M; F)\,, \ x \in \pa M
        \,, \ \mbox{ are invertible}\,.
\end{multline*}
All geometric differential operators associated to the metric $g_{sc}$
belong to $\Diff^m(\maV_{sc}; E, F) \subset \Psi^m(\maG_{sc}; E, F)$,
for suitable vector bundles $E, F \to M$.
\end{corollary}

We have $P_x = \pi_x(P)$ and $P_x$ is translation invariant (i.e.
constant coefficient, in this case), and hence can be studied using
the Fourier transform. In this case, all the vector fields yield ghost
derivatives.

\subsubsection{Asymptotically hyperbolic manifolds}

We continue to assume that $M$ is a manifold with smooth boundary $\pa
M$.  The groupoid $\maG_{ah}$ modeling asymptotically hyperbolic
spaces is chosen such that
\begin{equation*}
 Lie(\maG_{ah}) \simeq \Gamma(A(\maG_{ah})) \seq \maV_0 \ede \{X \in
 \Gamma(M; TM)\vert\, X\vert_{\pa M} = 0 \}\,.
\end{equation*}
Let $\maL_x := T_x(\pa M) \rtimes (0, \infty)$ be the semi-direct
product as in the previous subsection, with $(0, \infty)$ acting by
dilations on $T_x(\pa M)$. Let $\maL \to \pa M$ be the bundle of Lie
groups with fiber $\maL_x$. Then
\begin{equation*}
 \maG_{ah} \ede \maL \sqcup M_0 \times M_0 \tto \pa M \sqcup M_0\,.
\end{equation*}
The topology is again locally given by a {transformation}
groupoid. This can be seen in the case of $M = \RR^{n-1} \times [0,
  \infty)$ with the natural action of $G_n := \RR^{n-1} \rtimes (0,
  \infty)$ obtained by recalling that, as smooth manifolds, we have
  $\RR^{n-1} \rtimes (0, \infty) = \RR^{n-1} \times (0, \infty)$. This
  groupoid is a particular case of the {\em edge groupoid}, following
  next, so the reader can consult to the next section for more
  details.  We again have that $\maG_x = \maG_x^x =\maL_x$, if $x \in
  \pa M$, and that all derivatives are ghost derivatives at the
  boundary. The metric is
\begin{equation*}
 g_{ah} \ede x^{-2} h\,,
\end{equation*}
where $x$ is the distance to the boundary (close to the boundary) and
$h$ is an everywhere smooth metric on $M$, as before.

\begin{corollary}\label{cor.ah}
Let $s \in \RR$ and $P \in L_s^m(\maG_{ah}; E, F) \supset
\Psi^m(\maG_{ah}; E, F)$, with $M$ compact.  Denote $M_x := T_x \pa M
\times \RR.$ We have
\begin{multline*} %\label{eq.Fredholm.c2}
	P : H^s(M; E) \to H^{s-m}(M; F) \mbox{ is Fredholm}
        \ \ \Leftrightarrow \ \ P \mbox{ is elliptic and all}\\
	\ P_{x} : H^s(M_x; E) \to H^{s-m}(M_x; F)\,, \ x \in \pa M \,,
        \ \mbox{ are invertible}\,.
\end{multline*}
All geometric differential operators associated to the metric $g_{ah}$
belong to $\Diff^m(\maV_{ah}; E, F) \subset \Psi^m(\maG_{ah}; E, F)$,
for suitable vector bundles $E, F \to M$.
\end{corollary}

We thus need to study the invertibility of certain (right) invariant
operators on the group $G_{n-1} := \RR^{n-1} \rtimes (0, \infty)$, for
which standard methods of representation theory can be used.

\subsection{The edge groupoid}\label{ssec.edge}
This example is motivated by the results in \cite{Grieser, Grushin71,
  Mazzeo91, nistorDesing, SchulzeBook91}, where the original Fredholm
results on the edge calculus can also be found. It is a particular
case of the next example, that of a desingularization groupoid, but we
nevertheless treat it separately, for the benefit of the reader.  See
also \cite{dasguptaWongSG08, KrainerEdge14,VertmanHeatEdge15}. We
consider the following framework.

First, $M$ is a manifold with smooth boundary $\pa M$. We assume that
we are given a smooth fibration $\pi : \pa M \to B$, where $B$ is a
smooth manifold (thus without boundary). We fix a tubular neighborhood
$U$ of $\pa M$ in $M$: $U \simeq \pa M \times [0, 1)$.  Let $\maH := B
  \times B$ be the pair groupoid. We now construct the so called {\em
    edge groupoid} $\maG_e$ that will turn out to be a groupoid with
  Lie algebroid given by the set $\maV_e$ of vector fields on $M$ that
  are tangent to the fibers of $\pi : \pa M \to B$ (in particular,
  these vector fields are tangent to the boundary).  If $B=\pa M$,
  then we recover the groupoid that models asymptotically hyperbolic
  spaces.  We denote by $\Diff(\maV_e)$ the algebra of differential
  operators generated by $\maV_e$ and by multiplication with functions
  in $\CI(M)$.

Let $\maL := TB \rtimes \RR_+^* \to B$ be the bundle of Lie groups
obtained from $TB \to B$ (regarded as a bundle of commutative Lie
groups) by taking the semi-direct product with $\RR_+^* := (0,
\infty)$ acting by dilations on the fibers of $TB \to B$. Its pullback
$\pi\pullback(\maL)$ via $\pi : \pa M \to B$ is hence a Lie groupoid
with units $\pa M$. Let $M_0 := M \smallsetminus \pa M$ be the
interior of $M$. Then, as a set, the {\em edge groupoid} $\maG_e$ is
the disjoint union
\begin{equation}
 \maG_e \ede \pi\pullback (\maL)\ \sqcup\ (M_0 \times M_0) \tto \pa M
 \sqcup M_0 \ .
\end{equation}
To define the smooth structure on this groupoid, we could use either
the results in \cite{nistorJapan} or proceed directly in four steps as
follows.
\smallskip

{\em Step 1}.  We first consider the adiabatic groupoid $\maH_{ad}$ of
$\maH := B \times B$ (this is the {\em tangent groupoid} of
\cite{connesNCG}; see \cite{nistorDesing} for more details and
references).  The adiabatic groupoid $\maH_{ad}$ is a Lie groupoid
with units $B \times [0, \infty)$ and Lie algebroid $A(\maH_{ad}) = TB
  \times [0, \infty) \to B \times [0, \infty)$, which, as a vector
bundle, is the fibered pull-back of $A(\maH) = TB \to B$ to $B \times
[0, \infty)$ via the projection $B \times [0, \infty) \to B$.  The Lie
algebroid structure on the sections of $A(\maH_{ad})$ is not that of a
fibered pull-back Lie algebroid, but is given by $[X, Y](t) = t [X(t),
  Y(t)].$ As a set, $\maH_{ad}$ is the disjoint union
\begin{equation*}
  \maH_{ad} \, := \, A(\maH) \times \{0\} \, \sqcup \, \maH \times (0,
  \infty) \,.
\end{equation*}
The groupoid structure of $\maH_{ad}$ is such that $A(\maH) \times
\{0\}$ has the Lie groupoid structure of a bundle of Lie groups and
$\maH \times (0, \infty)$ has the product Lie groupoid structure with
$(0, \infty)$ the groupoid associated to a space (that is $(0,
\infty)$ has only units, and all orbits are reduced to a single
point). The smooth structure is obtained using the exponential
map. See also \cite{vanErpYuncken0}. \smallskip

{\em Step two.}\ Let $\pi : \pa M \to B$ be the given fibration
map. We denote also by $\pi$ the resulting map $\pa M \times [0,
  \infty) \to B \times [0, \infty)$. Then we consider the pullback Lie
    groupoid $\pi\pullback(\maH_{ad})$.
\smallskip

{\em Step three.}\ Let $\RR_{+}\sp{*} = (0, \infty)$ act by dilations
on the $[0, \infty)$ variable on $\pi\pullback (\maH_{ad})$ and
  consider the semi-direct product $\pi\pullback (\maH_{ad})\rtimes
  \RR_{+}\sp{*}$ \cite{DebordSkandalis}. As a set, it is the disjoint
  union of $ \pi\pullback (\maL)$ and of the pair groupoid of $\pa M
  \times (0, \infty)$.\smallskip

%%%%%%%%%%%%%%%%%%%%%%%%%%%%%%%%%%%%%%%%

{\em Step four.}\ Let us identify the tubular neighborhood $U \subset
M$ of $\pa M$ with $\pa M \times [0, 1)$. Then we can consider the
  reduction $\maH' := \big( \pi\pullback (\maH_{ad})\rtimes
  \RR_{+}\sp{*}\big )_U ^U$. By the previous step, this reduction
  $\maH'$ is the disjoint union of $ \pi\pullback (\maL)$ and of the
  pair groupoid of $\pa M \times (0, 1)$, which we can view as a
  subset of $M_0 \times M_0$, the pair groupoid of $M_0$. We then glue
  the reduction $\maH' := \big( \pi\pullback (\maH_{ad})\rtimes
  \RR_{+}\sp{*}\big )_U^U$ with $M_0 \times M_0$ by identifying the
  reduction of $\maH'$ to $U \smallsetminus \pa M$ (which is the pair
  groupoid of $\pa M \times (0, 1)$, as we have seen), with its image
  in $M_0 \times M_0$. This glueing construction is, of course,
  nothing but a particular case of the glueing construction in
  \cite{Gualtieri} and \cite{nistorDesing}.\smallskip

In any case, we obtain right away from the definition that the edge
groupoid $\maG_e$ is a Hausdorff \ssub\ groupoid.  The set of units of
$\maG_e$ is $M$ and the representations $\pi_x $, $x \in \pa M$, are
equivalent precisely when they map to the same point in $B$ and they
act on $\pi^{-1}(b) \times \maL_b \simeq \pi^{-1}(b) \times T_b B
\times \RR$.

If $P \in \Diff(\maV_e)$ and $b \in B$, we can obtain the limit
operators $P_b := \pi_b(P)$ as follows.  The restriction of $P$ to an
infinitesimal neighborhood of $\pi^{-1}(b)$ in $M$ will have some
hidden (ghost) derivatives coming from the Lie algebra of the group
$\maL_b := T_bB \rtimes \RR_+^\infty$.  Let us choose local
coordinates $(y, z)$ on $\pa M$ that are compatible with $\pa M \to
B$, in the sense that $z$ comes from a coordinate system on $B$. Let
$x$ denote the defining function of $\pa M$.  Then locally, $\maV_e$
is generated by $x \pa_x$, $\pa_{y_j}$, $x \pa_{z_k}$, with $j$ and
$k$ (always) ranging through a suitable index set:
\begin{equation}
 \maV_e \seq \CI(M) x \pa_x + \sum_j \CI(M)\pa_{y_j} + \sum_j \CI(M)
 x\pa_{z_j}\,.
\end{equation}

With this notation, the Lie algebra $\Lie(\maL_b)$ is generated by the
vector fields $x \pa_x$ and $x \pa_{z_k}$.  These come from non-zero
vector fields in $\maV_e$ that restrict to 0 on $\pi^{-1}(b)$, they
are the ghost derivatives. We let the ghost derivatives act on
$\maL_b$, and thus we obtain a differential operator $P_b$ on
$\pi^{-1}(b) \times \maL_b$.  The metric on $M$ is then an {\em edge
  metric} in the sense of Mazzeo \cite{Mazzeo91}. It can be obtained
by patching together metrics for which $x \pa_x$, $\pa_{y_j}$, $x
\pa_{z_k}$ are an orthonormal set of vectors.

Theorem \ref{thm.nonclassical2} then becomes.

\begin{corollary}\label{cor.e}
Let $P \in L_s^m(\maG_{e}; E, F) \supset \Psi^m(\maG_{e}; E, F)$, $M$
compact.  We have
\begin{multline*} %\label{eq.Fredholm.c2}
	P : H^s(M; E) \to H^{s-m}(M; F) \mbox{ is Fredholm}
        \ \ \Leftrightarrow \ \ P \mbox{ is elliptic and all}\\
	\ P_{b} : H^s(M_b; E) \to H^{s-m}(M_b; F)\,, \ b \in B \,,
        \ \mbox{ are invertible}\,,
\end{multline*}
where $M_b := \pi^{-1}(b) \times T_bB \times \RR.$ All geometric
differential operators associated to $g_e$ belong to
$\Diff^m(\maV_{e}; E, F) \subset \Psi^m(\maG_e; E, F)$, for suitable
$E, F \to M$.
\end{corollary}

If $B$ is reduced to a point and $\pa M$ is connected, the groupoid
$\maG_e$ constructed in the last subsection recovers the groupoid of
the $b$-calculus: $\maG_e = \maG_b$. It models in this case manifolds
with cylindrical ends. If $B = \pa M$, the corresponding groupoid
models ``asymptotically hyperbolic'' spaces.

\subsection{Desingularization groupoids}

One of the nice features of the class of \ssub\ groupoids is that it
is invariant with respect to desingularization along suitable
submanifolds.  In this subsection, we will give an {\em ad hoc}
argument for this statement.  Recall the fibered pull-back
$\pi\pullback (B)$ of Example \ref{ex.pullback}. We proceed in a
slightly greater generality than in \cite{nistorDesing}, to which we
refer for more details and for the unexplained arguments.

Let us assume that $M$ is a manifold with corners, that $H$ is a
hyperface of $M$ and that we are given a {\em tame} submersion of
manifolds with corners $\pi : H \to B$. Let $U \simeq H \times [0, 1)$
  be a tubular neighborhood of $H$ in $M$. The hyperface $H$ will play
  the role played by the boundary in the previous examples. Typically,
  $M$ will be the result of a desingularization procedure, such as a
  blow-up (see Remark \ref{rem.desing} and the following section).

We assume that we are given a Lie groupoid $\maG$ on $M \smallsetminus
H$ and a Lie groupoid $\maH \to B$ such that the following is
satisfied.\smallskip

\noindent {\bf Local fibered pull-back structure assumption:} {\em The
  reduction of $\maG$ to $H \times (0, 1)$ is isomorphic to the
  pullback $p\pullback (\maH)$ via the map $p := \pi \circ p_1: H
  \times (0, 1) \to B$.}
\smallskip

We fix $\maH$ as in the above assumption throughout this
subsection. We define then a groupoid $\maK \tto M$ such that, as a
set, it is the disjoint union
\begin{equation*}
 \maK \ede \pi\pullback (\maA(\maH )\rtimes \RR^*_+)\ \sqcup\ \maG
 \tto H \sqcup M \smallsetminus H .
\end{equation*}
The groupoid $\maH$ replaces the pair groupoid $B \times B$ of the
example of the edge groupoid. We proceed as in that example to
consider the adiabatic groupoid $\maH_{ad}$ of $\maH$, which has units
$B \times [0, \infty)$.  We pull back this groupoid to a groupoid
  $\pi\pullback(\maH_{ad})$ using the map $\pi_1 : H \times [0,
    \infty) \to B \times [0, \infty)$, and then we consider the
      semi-direct product $\pi_1\pullback (\maH_{ad}) \rtimes \RR^*_+
      = \pi_1\pullback (\maH_{ad} \rtimes \RR^*_+)$ with $\RR_+^*$
      acting by dilations on $[0, \infty)$.

As in the previous example of the edge groupoid, we view $U$ as an
open subset of $H \times [0, \infty)$ and we consider the groupoid
$\maH'$ defined as the reduction to $U$ of $\pi\pullback (\maH_{ad})
\rtimes \RR^*_+$, with $\maH$ as in the local fibered pull-back
assumption.  The construction is completed as in the fourth step of
the edge groupoid. That is, we use the invariant subset $H \times
\{0\}$ of the units of $\maH'$ to write the groupoid as a disjoint
union using Remark \ref{rem.disjoint}. The reduction of $\maH'$ to the
complement of $H \times \{0\}$, that is, to $H \times (0, 1)$, is
(isomorphic to) the fibered pull-back $p\pullback(\maH)$ of $\maH$ to
$H \times (0, 1)$, by the local fibered pull-back structure
assumption.  We can view this fibered pull-back $p\pullback(\maH)$ as
a subset of $\maG$. This gives that we can glue $\maH'$ and $\maG$
along the common open groupoid $p\pullback(\maH)$ to obtain a groupoid
$\maK$.

It can be proved that if $\maG$ is Fredholm, then $\maK$ is also
Fredholm. Moreover, if $\maG$ is a \ssub\ groupoid, then $\maK$ will
also be a \ssub\ groupoid.

\begin{remark}\label{rem.desing}
Typically, we start with a Lie groupoid $\maG'\tto M'$ and $L \subset
M'$ is a submanifold with corners. Then $M := [M':L]$, the blow-up of
$M'$ with respect to $L$ and $H$ is the hyperface corresponding to $L$
in this blow-up, with $\pi: H\to L$ the blow-down map. Finally, $\maG$
is the reduction of $\maG'$ to the complement of $L$. Then we denote
$[[ \maG': L]] := \maK$. This is the {\em desingularization}
construction from \cite{nistorDesing}.
\end{remark}

\subsection{Desingularization and singular spaces}\label{sec.ss}
Let us show how to use the desingularization construction in some
typical examples.  By $M_k$ and $L_k$ we will denote manifolds with
corners of depths $k$. Thus $M_0$ and $L_0$ will have, in fact, no
corners or boundary (hence they will be ``smooth'').

\subsubsection{The desingularization of a smooth submanifold}
\label{sssec.ex.smooth}
The simplest example of a desingularization groupoid 
is the desingularization of the pair groupoid 
$\maG_0 := M_0 \times M_0$ with respect to a smooth
submanifold $L_0 \subset M_0$. Recall that $M_0$ is also smooth. Thus 
neither $M_0$ nor $L_0$ have corners. The example of this
subsection is a particular case of the edge groupoid, and hence it is
related to the edge calculus.  \cite{Grushin71, Mazzeo91, SchulzeIE,
  SchulzeBook91}.

Let $N$ be the normal bundle of $L_0$ in $M_0$ and denote by $S
\subset N$ the set of unit vectors in $N$, that is, $S$ is the unit
sphere bundle of the normal bundle of $L$ in $M_0$.  We let $\pi : S
\to L_0$ be the natural projection. Then the blow-up $M_1 := [M_0 :
  L_0]$ of $M_0$ with respect to $L_0$ is the disjoint union
\begin{equation*}
 M_1 \, := \, [M_0 : L_0] \, := \, (M_0 \smallsetminus L_0) \sqcup S \,,
\end{equation*}
with the structure of a manifold with smooth boundary $S$.  We let
$\maG_1$ to be the associated edge groupoid introduced in Subsection
\ref{ssec.edge}. It is a Lie
groupoid with base $M_1 := [M_0: L_0]$. We are moreover in the
framework of Remark \ref{rem.desing}, so $\maG_1 = [[\maG_0:L_0]]$.
The filtration of $M_1$ has two sets, namely, $U_1 = M_1$ and $U_0 :=
M_0 \smallsetminus L_0 \subset M_1$, both of which are open and
invariant for $\maG_1$ (but $U_0$ is not invariant for $\maG_0$, in
general). 

It turns out that $\maG_1$ is a \ssub\ Lie groupoid
if we consider the two strata $U_0$ and $S$.  Also,
$(\maG_0)_{U_0}\sp{U_0} = U_0 \times U_0$ and $\maG_S$ is the fibered
pull-back of a bundle of Lie groups $\maL \to S$, as described in
Subsection \ref{ssec.edge}.

\subsubsection{The desingularization of a submanifold with boundary}
\label{ssec.bdry}
We now generalize he last construction to manifolds with boundary.
Let $M_1$ be a compact manifold with smooth boundary.  We denote by $F
:= \pa M_1$ its boundary and by $G := M_1 \smallsetminus F$ its
interior. On $M_1$ we consider the Lie algebra of vector fields
$\maV_b$ {\em tangent} to $\partial M_1$, as before. It is the space
of sections of $T\sp{b}M_1$, the ``$b$-tangent bundle''
\cite{MelroseAPS, MelroseScattering} of $M_1$.  Let $\maG_1 :=
\maG_b$, as defined in the previous section.

Let $L_1 \subset M_1$ be an embedded smooth submanifold assumed to be
such that its boundary is $\pa L_1 = L_1 \cap \pa M_1$ and such that
$L_1$ intersects $\pa M_1$ {\em transversely}.  Then $L_1$ has a
tubular neighborhood $U$ in $M_1$, and hence we can consider the
blow-up $M_2 := [M_1: L_1]$, which will be a manifold with corners of
codimension 2. Moreover, the reduction of $\maG_1 = \maG_b$ to $U
\smallsetminus L_1$ satisfies the local fibered pull-back structure
assumption, since it is the $b$-groupoid of $U \smallsetminus L_1$. In
view of Remark \ref{rem.desing}, we can then define $\maG_2 :=
[[\maG_1 : L_1]]$.

Again it turns out that $\maG_2$ is a \ssub\ Lie manifold; indeed,
this is seen by choosing the following filtration of $M_2$ with three
sets:
\begin{equation*}
    U_0 \, := \, M_1 \smallsetminus (L_1 \cup \pa M_1) \ \subset \ U_1
    \ede M_1 \smallsetminus L_1 \ \subset \ U_2 \, := \, M_2 \,.
\end{equation*}
The sets $U_j$ are open and $\maG_2$ invariant (but not $\maG_1$
invariant). The sets $U_0$ and $U_1$ are $\maG_2$-invariant by
construction. Assume they are connected, for simplicity. The
restriction $(\maG_2)_{U_1}$ coincides with the reduction
$(\maG_1)_{U_1}\sp{U_1}$ by the definition of the desingularization
groupoid (Remark \ref{rem.desing}). In particular, $(\maG_2)_{U_0} =
(\maG_2)_{U_0}\sp{U_0} = (\maG_1)_{U_0}\sp{U_0} = U_0^2$ and
\begin{equation*}
  (\maG_2)_{U_1 \smallsetminus U_0} \ede (\maG_1)_{U_1 \smallsetminus
    U_0} \sp{U_1 \smallsetminus U_0}\, = \, (U_1 \smallsetminus
  U_0)\sp{2} \times \RR_{+}\sp{*} \,,
\end{equation*}
where $U_0\sp{2}$ and $(U_1 \smallsetminus U_0)\sp{2}$ are pair
groupoids. In the general case, if $U_1 \smallsetminus U_0$ is not
connected, we write $U_1 \smallsetminus U_0 = \sqcup V_j$ as the
disjoint union decomposition of its connected components, then we have
$(\maG_2)_{U_1 \smallsetminus U_0} \ede \sqcup_j V_j\sp{2} \times
\RR_{+}\sp{*}$.  Let us denote as in the boundaryless case by $S$ the
unit sphere bundle of the normal bundle to $L_1$ in $M_1$. (The set
$S$ is the last stratum $U_2 \smallsetminus U_1$ in $M_2$.)  Then the
restriction of $\maG_1$ to $S := M_2 \smallsetminus U_1$ is the
following fibered fibered pull-back groupoid. Let $\pi : S \to L$ be
the natural projection, as before. Let again $G_{S} := T\sp{b} L
\rtimes \RR_{+}\sp{*} \to L$ be the group bundle over $L$ obtained by
taking the direct product of the $b$-tangent bundle to $L$ with the
action of $\RR_{+}\sp{*}$ by dilation on the fibers of $TL \to
L$. Then
\begin{equation*}
 \maG_S \ede \pi\pullback(G_S)\,.
\end{equation*}

\subsubsection{Desingularization of a stratified subset of dimension one}
We now deal with a slightly more complicated example by combining the
two previous examples. We thus introduce the groupoid that is obtained
from the desingularization of a stratified subset of dimension one.
Full details as well as applications will be included in a forthcoming
paper with Mihai Putinar.

Let $M_0$ be a smooth, compact manifold (so no corners). Let $L_0 :=
\{P_1, P_2, \ldots, P_k\}\subset M_0$ and let us assume that we are
given a subset $\maS \subset M_0$ such that
 \begin{equation}
 \maS \, = \, L_0 \, \cup \, \cup_{j=1}\sp{l} \gamma_j,
 \end{equation}
 where each $\gamma_j$ is the image of a smooth map $c_j : [0, 1] \to
 M_0$, with the following properties:
 \begin{enumerate}[(i)]
  \item $c_j'(t) \neq 0$,
  \item $c_j(0), c_j(1) \in L_0 := \{P_1, P_2, \ldots, P_k\}$,
  \item $c_j((0, 1))$ are disjoint and do not intersect $L_0$ and
  \item the vectors $c_j'(0)$ and $c_j'(1)$, $j = 1, \ldots, l$,
    are all distinct.
 \end{enumerate}
 
We now introduce the desingularization of $\maG_0 := M_0 \times M_0$
(the pair groupoid) with respect to $\maS$. The set $L_0 := \{P_1,
P_2, \ldots, P_k\} \subset M_0$ satisfies the assumptions of
\ref{sssec.ex.smooth}. We can first define $\maG_1$ to be the
desingularization of $\maG_0$ with respect to $L_0$ as in that
example:
\begin{equation}
 \maG_1 \, := \, [[\maG_0 : L_0]] \,,
\end{equation}
which is a groupoid with units $M_1 := [M_0: L_0]$.  The smooth maps
$c_j$ then lift to smooth maps
\begin{equation*}
 \tilde c_j : [0, 1] \to [M_0: L_0]\,.
\end{equation*}
The assumption that the vectors $c_j'(0)$ and $c_j'(1)$, $j = 1,
\ldots, l$, are all distinct then gives that the sets $\tilde \gamma_j
:= \tilde c_j([0, 1])$ are all disjoint and intersect the boundary of
$M_1$ transversally. Let $L_1$ be the disjoint union of the embedded
curves $\tilde \gamma_j$. Then we can perform a further
desingularization along $L_1$, as in \ref{ssec.bdry}, thus obtaining
\begin{equation*}
 \maG_2 \, := \, [[\maG_1 : L_1]] \,,
\end{equation*}
which is a boundary fibration Lie groupoid. The Lie groupoid $\maG_2$
is the {\em desingularization} groupoid of $M_0$ with respect to
$\maS$.  Its structure is given as in the previous subsection.

This example can be extended to higher dimensional cases by using
clean intersecting families.

\subsubsection*{Acknowledgements} We thank Vladimir Georgescu,
Marius Mantoiu, and Wolfgang Schulze for useful discussions.

\bibliographystyle{plain}

\def\cprime{$'$}

\end{document}